\documentclass[a4paper,12pt]{amsart}

\usepackage{amsmath,amsthm,amscd}
\usepackage{amssymb}
\usepackage[ansinew]{inputenc}
\usepackage[T1]{fontenc}
\usepackage[dvipsnames]{xcolor}
\usepackage{rotating}
\usepackage[pdfstartpage=1]{hyperref}
\hypersetup{
    pdfview={Fit},    
    pdfstartview={FitH}
}

\newtheorem{thm}{Theorem}

\newtheorem{lem}[thm]{Lemma}

\newtheorem{prop}[thm]{Proposition}

\newtheorem{defn}{Definition}

\theoremstyle{definition}

\newtheorem{notation-definition}{Notation-definition}[section]
\newtheorem{exmp}[thm]{Example}
\newtheorem{rem}[thm]{Remark}
\newtheorem*{notation}{Notation}

  \newcommand{\cH}{\mathcal{H}} 
\newcommand{\cJ}{\mathcal{J}}   
  \newcommand{\cO}{\mathcal{O}} 
  
\newcommand{\fa}{\mathfrak{a}} \newcommand{\fb}{\mathfrak{b}} 
  \newcommand{\fm}{\mathfrak{m}}
 \newcommand{\fp}{\mathfrak{p}} 
 
 \newcommand{\bbN}{\mathbb{N}}

\newcommand{\ga}{\alpha}

\newcommand{\lrarrow}{\longrightarrow}

  \DeclareMathOperator{\Spec}{Spec}

\title{A formula for jumping numbers in a two-dimensional regular local ring}

\author{Eero Hyry}
\address{
Mathematics and Statistics\\
Faculty of Natural Sciences\\ 
University of Tampere\\
FIN-33014 Tampereen yliopisto\\ 
Finland}
\email{eero.hyry@uta.fi}
  
\author{Tarmo J\"arvilehto}
\address{
P\"a\"askykuja 5\\
FI-04620 M\"ants\"al\"a\\ 
Finland}
\email{tarmo.jarvilehto@pp.inet.fi}

\date{}

\begin{document}

\begin{abstract}In this article we give an explicit formula for the jumping numbers of an ideal of finite 
colenght in a two-dimensional regular local ring with an algebraically closed residue field. 
For this purpose, we associate a certain numerical semigroup to each vertex of the dual 
graph of a log-resolution of the ideal.
\end{abstract}

\maketitle

\section{Introduction}

Jumping numbers measure the complexity of the singularities of a closed subscheme 
of a variety. They are defined in terms of multiplier ideals of the subscheme. Multiplier 
ideals form a nested sequence of ideals parametrized by rational numbers. The values 
of the parameter where the multiplier ideal changes are called jumping numbers. For a 
simple complete ideal in a local ring of a closed point on a smooth surface an explicit 
formula has been provided by J\"arvilehto in~\cite{J}, which is based on the dissertation ~\cite{Jj}. 
This result applies also to jumping numbers of an analytically irreducible plane curve. 
The purpose of this article is to generalize this formula to any complete ideal.  

Besides~\cite{J}, jumping numbers of simple complete ideals or analytically irreducible 
plane curves have been independently investigated by several people (see~\cite{N}, \cite{T0} and~\cite{GM}). 
In a local ring at a rational singularity of a complex surface, Tucker presented in~\cite{T} 
an algorithm to compute the set of jumping numbers of any ideal. Recently, Alberich-Carrami\~nana, 
Montaner and Dachs-Cadefau gave in~\cite{AM} another algoritm for this purpose. 
But even in dimension two finding a closed formula for the general case has turned out
to be difficult. Kuwata calculated  in~\cite{K} the smallest jumping number, the so called 
log-canonical threshold, for a reduced plane curve with two branches. Galindo, Hernando 
and Monserrat succeeded in~\cite{GHM} to generalize this to any number of branches.  

Jumping numbers are defined by using an embedded resolution of the subscheme. 
They depend on the exceptional divisors appearing in the resolution. We therefore look at 
the dual graph of the resolution. Recall that the vertices of the dual graph correspond to 
exceptional divisors and two vertices are connected by an edge if the exceptional divisors 
in question intersect. To each vertex, we will attach a certain semigroup. We will then 
describe jumping numbers in terms of these semigroups. In defining the semigroups we 
use Zariski exponents of the valuations associated to the exceptional divisors. 

To explain this in more detail, let $\fa$ be a complete ideal of finite colength in a two-dimensional 
regular local ring $R$ having an algebraically closed residue field. Let $X\lrarrow \Spec R$ 
be a log resolution of the pair $(R,\fa)$. Let $E_1,\ldots,E_N$ be the exceptional divisors. 
Let $\Gamma$ be the dual graph of $X$. Two vertices $\gamma$ and $\eta$ are called adjacent, denoted by 
$\gamma\sim\eta$, if  the corresponding exceptional divisors $E_\gamma$ and $E_\eta$ intersect. The valence $v_\Gamma(\mu)$ 
of a vertex $\mu$ means the number of vertices adjacent to it.  A vertex with valence at most 
one is called an end whereas a vertex of valence at least three is a star. Let $v_1,\ldots,v_N$ 
be the discrete valuations and $\fp_1,\ldots,\fp_N$ the simple ideals corresponding to $E_1,\ldots,E_N$, 
respectively. Set $V_{\mu,\nu}=v_{\mu}(\fp_\nu)$ for all $\mu,\nu=1,\ldots,N$. The Zariski exponents are the 
numbers $V_{\mu,\tau}$, where $\tau$ is end. Let $S^\mu$ denote the submonoid of $\mathbb N$ generated by $V_{\mu,\mu}$ and 
the numbers 
$$
s^\mu_\nu:=\gcd\left\{V_{\mu,\tau}\mid v_\Gamma(\tau)=1
	\text{ and }
\tau\in\Gamma_\nu^\mu\right\},
$$
where $\Gamma^\mu_\nu$ is the branch emanating from $\mu$ towards $\nu$, i.~e., the maximal connected subgraph 
of $\Gamma$ containing $\nu$ but not $\mu$. We will show that $S^\mu$ is a numerical semigroup 
generated by at most two elements (see Remark \ref{Sgen2}).

Recall that a divisor $F=f_1E_1+\ldots +f_NE_N$ on $X$ is called antinef if $F \cdot E_\gamma \le 0$ 
for all $\gamma=1,\ldots,N$, where $F \cdot E_\gamma$ is the intersection product. Let $\{\widehat E_1,\dots,\widehat E_N\}$, where 
$E_\mu\cdot\widehat E_\nu=-\delta_{\mu,\nu}$, denote the dual basis of $\{E_1,\dots,E_N\}$. Then $F=\widehat f_1\widehat E_1+\ldots+\widehat f_N\widehat E_N$
is antinef if and only if $\widehat f_i\ge 0$ for all $i=1,\ldots,N$.  We call the numbers $\widehat f_1,\ldots, \widehat f_N$ as the 
factors of $F$.

We make use of the observation made in~\cite{J} that jumping numbers of $\fa$ can be 
parametrized by the antinef divisors. More precisely, the jumping number corresponding 
to an antinef divisor $F$ is 
$$
\xi_F:=\min_\gamma \frac{f_\gamma+k_\gamma+1}{d_\gamma},
$$ 
where $D=d_1E_1+\ldots+ d_NE_N$ is the divisor on $X$ such that $\cO_X(-D)=\fa \cO_X$, and 
$K=k_1E_1+\ldots+ k_NE_N$ denotes the canonical divisor. We say that $\xi$ is a jumping 
number supported at a vertex $\mu$ if $\xi=\xi_F$ for some antinef divisor $F$ with 
$$
\xi_F =\frac{f_\mu+k_\mu+1}{d_\mu}.
$$ 
We fix a vertex $\mu$ and concentrate on the set $\mathcal H_\mu^\fa$ of jumping numbers supported at $\mu$. 

Our main result, Theorem~\ref{2}, yields a formula for the set of the jumping numbers of 
$\fa$ supported at $\mu$:
$$
\cH_\mu^\fa
=\left\{\frac{t}{d_\mu}\middle | t+(v_{\Gamma}(\mu)-2)V_{\mu,\mu}
	-\sum_{\nu\sim\mu}s^\mu_\nu\left\lceil t\sum_{i\in\Gamma_\nu^\mu}\frac{\widehat d_i V_{\mu,i}}{s^\mu_\nu d_\mu}\right\rceil^{+}\in S^\mu\right\},
$$
where $\left\lceil\:\:\right\rceil^{+}$ means rounding up to the nearest positive integer. Note that a jumping 
number is always supported at a vertex which is either a star or corresponds to a simple 
factor of the ideal (see Lemma~\ref{3d}). 

In the proof we look at the factors of divisors. Given a vertex $\mu$ we introduce two 
transforms of divisors by means of which it is possible 'bring' factors from each branch 
emanating from $\mu$ to the closest vertex adjacent to $\mu$ and 'distribute' a part of a factor 
from $\mu$ to the adjacent vertices. Suppose that $\xi=\xi_F$ is supported at a $\mu$. Using these 
transformations we can modify either $F$ or $D$ or both in such a way that we still have 
$\xi=\xi_F$. In particular, we can assume that the divisor $D$ has factors only at the vertices 
adjacent to $\mu$. In this process the properties of the mappings $\rho_{[\mu,\gamma]}\colon \Gamma\to\mathbb Q$, where $\mu$ and 
$\gamma$ are fixed vertices, and 
$$
\nu\mapsto\frac{V_{\gamma,\nu}}{V_{\mu,\nu}},
$$ 
play a crucial role. In particular, we prove in Lemma~\ref{V<V} that $\rho_{[\mu,\gamma]}$ is strictly increasing 
along the path going from $\mu$ to $\gamma$, and stays constant on any path going away from this 
path. Finally, we  show in Example~\ref{generalexample} how our formula works in practice.

\section{Preliminaries}

In this paper, we make use of the Zariski-Lipman theory of complete ideals. The 
general setting here is similar to that discussed in our paper \cite{HJ}. For the reader's convenience, 
we collect here some basic concepts and notation. More details can be found 
in~\cite{L}, \cite{C}, \cite{LW} and~\cite{J}.

\subsubsection*{About Zariski-Lipman theory}

Let $\fa$ be a complete ideal of finite colength in a two-dimensional regular local ring $R$ 
having an algebraically closed residue field. Let $\pi \colon X\rightarrow\Spec(R)$ be a principalization 
of $\fa$. Then $X$ is a regular scheme and $\fa\cO_X=\cO_X(-D)$ for an effective Cartier divisor 
$D$. The morphism $\pi$ is a composition of point blowups of regular schemes 
\begin{equation*}
\pi:X=X_{N+1}\xrightarrow{\pi_N}\cdots\xrightarrow{\pi_2} X_2\xrightarrow{\pi_1}X_1=\Spec R,
\end{equation*} 
where $\pi_\mu$ is the blowup of $X_\mu$ at a closed point $x_\mu\in X_\mu$. Let $E_\mu$ be the strict and $E^*_\mu$ the 
total transform of the exceptional divisor $\pi^{-1}_\mu\{x_\mu\}$ on $X$. We write $v_\mu$ for the discrete 
valuation associated to the discrete valuation ring $\cO_{X,E_\mu}$, so that $v_\mu$ is the $\fm_{X_\mu,x_\mu}$-adic 
order valuation.

A point $x_\mu$ is \textit{infinitely near} to a point $x_\nu$, if the projection $X_\mu \rightarrow X_\nu$ maps $x_\mu$ to 
$x_\nu$. Further, $x_\mu$ is \textit{proximate} to $x_\nu$, denoted by $\mu\succ\nu$, if $x_\mu$ lies on the strict transform 
of $\pi_\nu^{-1}\{x_\nu\}$ on $X_\mu.$ Note that a point can be proximate to at most two points. The 
\textit{proximity matrix} is
\begin{equation*}\label{prox}
P:=(p_{\mu,\nu})_{N\times N},\text{ where }p_{\mu,\nu}=
\left\lbrace
\begin{array}{rl}
   1,&\text{if }\mu=\nu;\\
  -1,&\text{if }\mu\succ \nu;\\
   0,&\text{otherwise.}
\end{array}
\right.
\end{equation*}
We write $Q=(q_{\mu,\nu})_{N\times N}:=P^{-1}$, so that $PQ=1$. 

Besides the obvious one, the lattice $\Lambda:= \mathbb Z E_1 \oplus  \ldots \oplus \mathbb Z E_N$ of exceptional divisors 
on $X$
has two other convenient bases, namely $\{E_1^*,\dots,E_N^*\}$ and $\{\widehat E_1,\dots,\widehat E_N\}$, where  
$E_\mu\cdot\widehat E_\nu=E_\mu^*\cdot E_\nu^*=-\delta_{\mu,\nu}$. Throughout this paper we use the practice that if an 
upper case letter, say $G$, denotes a divisor $G\in\Lambda$, then the corresponding lower case 
letter possibly with an accent mark denotes the coefficient vector with respect to the 
appropriate base. In particular, writing
$$
G=g_1E_1+\ldots+g_NE_N=g_1^*E^*_1+\ldots+g_N^*E^*_N=\widehat g_1\widehat E_1+\ldots+\widehat g_N\widehat E_N
$$
with $g=(g_\nu), g^*=(g^*_\nu)$ and $\widehat g=(\widehat g_\nu)$, we get the following base change formulas:
\begin{equation}\label{BC}
g^*=gP^\textsc{t}\text{ and }\widehat g=gP^\textsc{t}P=g^*P.
\end{equation}
In many cases, we regard $\Lambda$ as a subset of $\Lambda_\mathbb Q := \mathbb Q\otimes\Lambda$. We call the vector $\widehat g$ the 
\emph{factorization vector} and $g$ the \emph{valuation vector} of the divisor. Note that $g=\widehat gV$, where 
$V:=(P^\textsc{t}P)^{-1}$ is called the \emph{valuation matrix}. Set
$$
w_\Gamma(\mu):=-E_\mu^2=1+\#\{\nu\mid \nu \succ \mu\}.
$$
We then get the formulas
\begin{equation}\label{P}
\widehat g_\mu=g^*_\mu-\sum_{\nu\succ \mu}g^*_\nu=w_\Gamma(\mu)g_\mu-\sum_{\nu\sim\mu}g_\nu\quad(\mu=1,\ldots,N).
\end{equation}
Especially, this yields
\begin{equation}\label{PW}
w_\Gamma(\eta)V_{\mu,\eta}=\sum_{i\sim\eta}V_{\mu,i}+\delta_{\mu,\eta}.
\end{equation}

Recall that a divisor $F\in \Lambda$ is \textit{antinef} if $\widehat f_\nu=-F \cdot E_\nu \ge 0$ for all $\nu=1,\ldots,N$.
Equivalently, the \textit{proximity inequalities}
\begin{equation*}
f^*_\mu\ge\sum_{\nu\succ\mu}f^*_\nu\quad(\mu=1,\ldots,N)
\end{equation*}
hold. Note that they can also be expressed in the form
\begin{equation*}
w_\Gamma(\mu)f_\mu\ge\sum_{\nu\sim\mu}f_\nu\quad(\mu=1,\ldots,N).
\end{equation*}
In fact, if $F\not=0$ is antinef, then also $f_\nu>0$ for all $\nu=1,\ldots,N$. There is a one to 
one correspondence between the antinef divisors in $\Lambda$ and the complete ideals of finite 
colength in $R$ generating invertible $\cO_X$-sheaves, given by $F\leftrightarrow\Gamma(X,\cO_X(-F))$. 

An ideal is called \textit{simple} if it cannot be expressed as a product of two proper ideals. 
By the famous result of Zariski, every complete ideal factorizes uniquely into a product 
of simple complete ideals. More precisely, we can present a complete ideal $\fa$ as a product
$$
\fa=\fp_1^{\widehat d_1}\cdots\fp_N^{\widehat d_N},
$$ 
where $\fp_\mu\subset R$ denotes the simple complete ideal of finite colength corresponding to the 
exceptional divisor $\widehat E_\mu$ and $\widehat d_i\in\mathbb N$ for every $i$. By~\eqref{BC}
\begin{equation*}
\widehat E_\mu=\sum_\nu q_{\mu,\nu}E^*_\nu=\sum_{\nu,\rho}q_{\nu,\rho}q_{\mu,\rho}E_\nu.
\end{equation*} 
In particular, we observe the \textit{reciprocity formula} 
\begin{equation*}
v_\nu(\fp_\mu)=\sum_{\rho}q_{\nu,\rho}q_{\mu,\rho}=v_\mu(\fp_\nu)\quad(\mu,\nu=1,\ldots,N),
\end{equation*} 
in short, $V=V^\textsc t$. 
Recall that the \textit{canonical divisor} is $K=\sum_\nu E^*_\nu$. If $k =(k_\nu)$ and 
$\widehat k =(\widehat k_\nu)$ are the appropriate coefficient vectors, we have
$$
kE=\widehat k\widehat E=K.
$$
The formulas \eqref{BC} yield 
\begin{equation}\label{K}
k_\nu=\sum_\mu q_{\nu,\mu}
	\quad\hbox{and}\quad
\widehat k_\nu=E_\nu^2+2\quad(\nu=1,\ldots,N).
\end{equation} 

\subsubsection*{Dual graph}

The \textit{dual graph} $\Gamma$ associated to our principalization is a tree, where the vertices 
correspond one to one to the exceptional divisors and an edge between two adjacent
vertices, $\gamma\sim\eta$, means that the corresponding exceptional divisors  $E_\gamma$ and $E_\eta$ intersect. 
A vertex $\gamma$ corresponding to the exceptional divisor $E_\gamma$ is weighted by the number $w_\Gamma(\gamma)$.
We say that a vertex $\gamma$ is proximate to another vertex $\eta$ if $p_{\gamma,\eta}=-1$. It is \textit{free} if it is 
proximate to at most one vertex. We may also say that $\gamma$ is infinitely near to $\eta$, and 
write $\eta\subset\gamma$, if this is the case with the corresponding points. The \textit{root} of $\Gamma$ is the vertex 
$\tau_0$ for which $\tau_0\subset\gamma$ for every $\gamma\in\Gamma$.

Blowing up a point on $E_\gamma$ expands the dual graph by adding a vertex $\nu$ corresponding 
to the exceptional divisor of the blowup. The weight of the new vertex is one and 
the weights of the adjacent vertices are increased by one. In \cite[Definition 5.1]{S1} such 
expansions are called elementary modifications. There are two kinds of elementary modifications. 
If $E_\gamma$ is the only exceptional divisor containing the center of blowup so that 
$\gamma\sim\nu$ forms the only new edge, then the elementary modification is of the first kind:
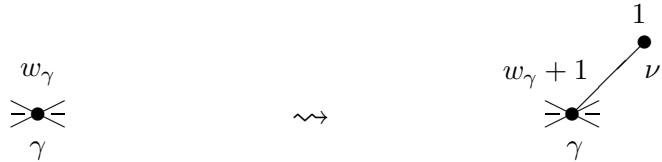
\newline
\begin{picture}(400,75)(-40,-25)
\put(35,0){\circle*{5}}
\put(35,15){\makebox(0,0)[cc]{\small$w_\gamma$}}
\put(35,-14){\makebox(0,0)[cc]{\small$\gamma$}}
\put(35,0){\line(-2,1){10}}
\put(35,0){\line(-2,-1){10}}
\put(30,0){\line(-1,0){5}}
\put(35,0){\line(2,1){10}}
\put(35,0){\line(2,-1){10}}
\put(40,0){\line(1,0){5}}
\put(235,0){\line(-2,1){10}}
\put(235,0){\line(-2,-1){10}}
\put(230,0){\line(-1,0){5}}
\put(235,0){\line(2,1){10}}
\put(235,0){\line(2,-1){10}}
\put(240,0){\line(1,0){5}}
\put(130,-5){\large$\leadsto$}
\put(235,0){\circle*{5}}
\put(225,15){\makebox(0,0)[cc]{\small$w_\gamma+1$}}
\put(236,-14){\makebox(0,0)[cc]{\small$\gamma$}}
\put(235,0){\line(1,1){27}}
\put(262,27){\circle*{5}}
\put(265,16){\makebox(0,0)[cc]{\small$\nu$}}
\put(260,38){\makebox(0,0)[cc]{\small$1$}}
\end{picture}

If the center of blowup is the intersection point of $E_\gamma$ and another exceptional divisor, 
say $E_\eta$, then the edge $\gamma\sim\eta$ is replaced by the edges $\gamma\sim\nu$ and $\nu\sim\eta$, and the elementary 
modification is of the second kind:
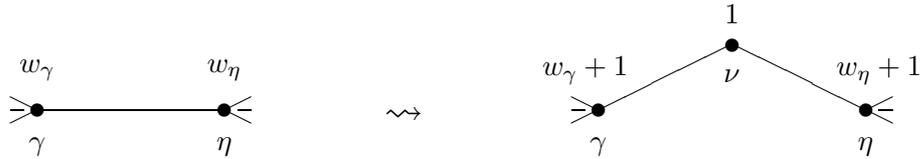
\newline
\begin{picture}(400,75)(-40,-25)
\put(0,0){\circle*{5}}
\put(0,15){\makebox(0,0)[cc]{\small$w_\gamma$}}
\put(0,-14){\makebox(0,0)[cc]{\small$\gamma$}}
\put(0,0){\line(1,0){70}}
\put(0,0){\line(-2,1){10}}
\put(0,0){\line(-2,-1){10}}
\put(-5,0){\line(-1,0){5}}
\put(70,0){\line(2,1){10}}
\put(70,0){\line(2,-1){10}}
\put(75,0){\line(1,0){5}}
\put(210,0){\line(-2,1){10}}
\put(210,0){\line(-2,-1){10}}
\put(205,0){\line(-1,0){5}}
\put(310,0){\line(2,1){10}}
\put(310,0){\line(2,-1){10}}
\put(315,0){\line(1,0){5}}
\put(70,0){\circle*{5}}
\put(70,15){\makebox(0,0)[cc]{\small$w_\eta$}}
\put(70,-14){\makebox(0,0)[cc]{\small$\eta$}}
\put(130,-5){\large$\leadsto$}
\put(210,0){\circle*{5}}
\put(205,15){\makebox(0,0)[cc]{\small$w_\gamma+1$}}
\put(210,-14){\makebox(0,0)[cc]{\small$\gamma$}}
\put(210,0){\line(2,1){50}}
\put(310,0){\line(-2,1){50}}
\put(260,24.5){\circle*{5}}
\put(260,13){\makebox(0,0)[cc]{\small$\nu$}}
\put(260,36){\makebox(0,0)[cc]{\small$1$}}
\put(310,0){\circle*{5}}
\put(315,15){\makebox(0,0)[cc]{\small$w_\eta+1$}}
\put(310,-14){\makebox(0,0)[cc]{\small$\eta$}}
\end{picture}

Let us write 
$$
\Gamma(\nu,U)\text{ where }U=\{\gamma\in\Gamma\mid\gamma\prec\nu\}
$$
for an elementary modification of the graph $\Gamma$ by adding a vertex $\nu$ adjacent to vertices 
$\gamma\in U$. Note that $U$ consists of at most two vertices. Note also that if the graph is empty 
then the elementary modification is defined to be of the first kind containing only the 
root vertex. Following \cite[Definition 5.2]{S1}, a dual graph dominates a dual graph $\Gamma$, if it 
can be obtained from $\Gamma$ by a sequence of elementary modifications. Obviously, a sequence 
of point blowups correspond to a sequence of elementary modifications. Especially, the 
dual graph of our principalization can be obtained from  the graph containing only the 
root vertex through successive elementary modifications (c.f. \cite[Remark 5.5]{S1}).

In a way, the matrix $P^\textsc {t}P$ represents the dual graph because the diagonal elements 
$(P^\textsc{t}P)_{\nu,\nu}=-E_\nu^2$ correspond with the weights of the vertices while outside the diagonal 
the element $(P^\textsc{t}P)_{\mu,\nu}=-E_\mu\cdot E_\nu$ is $-1$ if $E_\mu$ and $E_\nu$ intersect and otherwise zero. 

The \textit{valence} $v_\Gamma(\mu)$ of a vertex $\mu$ means the number of vertices adjacent to it. If 
$v_\Gamma(\mu)\ge 3$, then $\mu$ is called a \textit{star}. If $v_\Gamma(\mu)\le1$, then we call it an \textit{end}. The vertices 
adjacent to $\mu$ correspond one to one to the branches emanating from $\mu$, which can be 
defined as follows:

\begin{defn} 
For any two vertices $\mu$ and $\nu$ in $\Gamma$, let $\Gamma^\mu_\nu$ denote the maximal connected 
subgraph of $\Gamma$ containing $\nu$ but not $\mu$ $(\Gamma_\mu^\mu=\emptyset)$. We say $\Gamma^\mu_\nu$ is a branch emanating from 
$\mu$ towards $\nu$. A branch $\Gamma^\mu_\nu$ is anterior to $\mu$, if $\mu$ is infinitely near to some of its vertices. 
Otherwise we say it is posterior to $\mu$.
\end{defn}

Observe that every branch emanating from $\mu$ is either anterior or posterior to $\mu$, and 
for those we immediately get the following result:

\begin{prop} 
The unempty posterior branches of $\mu$ correspond one to one to the free 
vertices, which are proximate to $\mu$, whereas the anterior branches are in one to one 
correspondence with the vertices to which $\mu$ is proximate to. 
\end{prop}

\begin{proof} 
The claim is trivial if $\mu$ is the only vertex, meaning that there are no unempty 
branches. We shall proceed by induction on the number of vertices. Suppose $\Gamma=\Gamma'(\eta,U)$ 
and the claim holds for $\Gamma'$. Observe that for any $\mu\in\Gamma$, $\mu$ is not proximate to $\eta$.

If $\mu=\eta$, then $\gamma\prec\mu$ exactly when $\gamma\sim\mu$, so that $\mu$ is infinitely near to any adjacent 
vertex. Obviously, the branches $\Gamma^\mu_\gamma$ correspond one to one to the vertices $\gamma\sim\mu$. Thus 
the claim is clear in this case.

Suppose that $\mu\neq\eta$. If $\eta$ is not a free vertex proximate to $\mu$, then the blowup just 
augments an existing branch of $\Gamma'$, i.~e., the branches of $\Gamma$ emanating from $\mu$ correspond
one to one to those of $\Gamma'$. Because the proximity relations are preserved under blowup, 
the claim follows. If $\eta$ is a free blowup of $\mu$, then $\mu\prec\eta$ and $\Gamma_\eta^\mu=\{\eta\}$ forms a new 
branch, which corresponds to the vertex $\eta$. For the rest of the branches emanating from 
$\mu$ the correspondence is inherited from $\Gamma'$.
\end{proof} 

Recall that a vertex is proximate to at most two vertices. Subsequently, there are at 
most two branches anterior to $\mu \in \Gamma$ depending on whether $\mu$ is free or not.

The \textit{distance} between two vertices $\mu, \nu\in \Gamma$ is defined as the length of the \textit{path} $[\nu,\mu]$, 
i.~e.,
$$
d(\nu,\mu):=\min\{r\mid\nu=\nu_0\sim\cdots\sim\nu_r=\mu,
	\text{ where }
\nu_0,\dots,\nu_r\in\Gamma\},
$$
Furthermore, if $T\subset \Gamma$, we set 
$$
d(\nu,T):=\min\{d(\nu,\mu)\mid\mu\in T\}.
$$ 
If $d(\nu,T)=1$, then we write $\nu\sim T$.

\begin{defn}\label{pair}
A pair $(\gamma,\tau)$ is associated to $\mu$, if $\gamma$ and $\tau$ satisfy the following three 
conditions:
\begin{itemize}
	\item[i)] $\gamma\subset\tau\subset\mu$, i.~e., $\mu$ is infinitely near to $\tau$ which is infinitely near to $\gamma$;
	\item[ii)] $\tau$ is free and infinitely near to every free vertex $\nu\subset\mu$;
	\item[iii)] $\gamma$ is not free and infinitely near to every non free vertex $\nu\subset\tau$, unless every $\nu\subset\tau$ 
	            is free in which case $\gamma=\tau_0$ is the root.
\end{itemize}
The sequence of pairs $((\gamma_i,\tau_{i+1}))_{i=0}^{g}$ is associated to $\mu:=\gamma_{g+1}$, if it holds for $i=0,\dots,g$
that $(\gamma_i,\tau_{i+1})$ is the pair associated to $\gamma_{i+1}$.
\end{defn}

\begin{rem}\label{rem2}
Let $\Gamma$ be the dual graph of $\mu$, i.~e., the simple dual graph which consists 
of all the vertices to which $\mu$ is infinitely near to. Observe that we may always reach 
this situation by repeatedly blowing down any vertex different from $\mu$ having a weight one. If 
$((\gamma_i,\tau_{i+1}))_{i=0}^{g}$ is now the sequence associated to $\mu$, then $\gamma_0=\tau_0$ is the root, $\tau_0,\dots,\tau_{g+1}$ 
are exactly the end vertices of $\Gamma$ while $\gamma_1,\dots,\gamma_g$ are its stars (cf. \cite[Proposition 4.3]{J}).
Note that the integer $g$, i.~e., the number of star vertices of the dual graph, is denoted 
by $g^*$ in \cite[Notation 3.3]{J}.
\end{rem}

\begin{rem}\label{rem2kuva}
As the relation $\nu\subset\mu$ induces a partial order on $\Gamma$, we might give the 
definition as follows: a pair $(\gamma,\tau)$ is associated to $\mu$, if $\tau$ is maximal among the free 
points to which $\mu$ is infinitely near to, and $\gamma$ is maximal among the non free points to 
which $\tau$ is infinitely near to. The graph below illustrates an example of a sequence of 
pairs associated to a vertex.

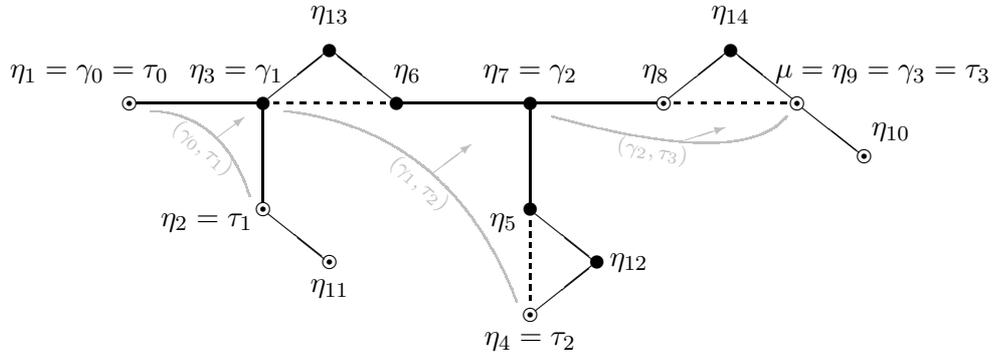
\begin{picture}(350,154)(-37,-104)
\put(0,0){\circle{5}}
\put(0,0){\circle*{2}}
\put(50,0){\circle*{5}}
\put(50,-40){\circle{5}}
\put(50,-40){\circle*{2}}
\put(100,0){\circle*{5}}
\put(150,-40){\circle*{5}}
\put(150,-80){\circle{5}}
\put(150,-80){\circle*{2}}
\put(150,0){\circle*{5}}
\put(200,0){\circle{5}}
\put(250,0){\circle{5}}
\put(200,0){\circle*{2}}
\put(250,0){\circle*{2}}
\put(275,-20){\circle{5}}
\put(275,-20){\circle*{2}}
\put(75,-60){\circle{5}}
\put(75,-60){\circle*{2}}
\put(175,-60){\circle*{5}}
\put(75,20){\circle*{5}}
\put(225,20){\circle*{5}}
\thicklines
\put(2.5,0){\line(1,0){47.5}}
\multiput(54,0)(5,0){9}{\line(1,0){2}}
\put(100,0){\line(1,0){50}}
\put(150,0){\line(1,0){47.5}}
\multiput(204,0)(5,0){9}{\line(1,0){2}}
\put(50,0){\line(0,-1){37.5}}
\put(150,0){\line(0,-1){37.5}}
\multiput(150,-44.5)(0,-4){8}{\line(0,-1){2}}
\thinlines
\put(51.9,-41.7){\line(5,-4){21.1}}
\put(51.9,1.7){\line(5,4){21.1}}
\put(98.1,1.7){\line(-5,4){21.1}}
\put(201.9,1.7){\line(5,4){21.1}}
\put(248.1,1.7){\line(-5,4){21.1}}
\put(151.9,-41.7){\line(5,-4){21.1}}
\put(151.9,-78,3){\line(5,4){21.1}}
\put(251.9,-1.7){\line(5,-4){21.1}}
\put(-15,11){\makebox(0,0)[cc]{\small$\eta_1=\gamma_0=\tau_0$}}
\put(29,-45){\makebox(0,0)[cc]{\small$\eta_2=\tau_1$}}
\put(40,11){\makebox(0,0)[cc]{\small$\eta_3=\gamma_1$}}
\put(150,-90){\makebox(0,0)[cc]{\small$\eta_4=\tau_2$}}
\put(140,-45){\makebox(0,0)[cc]{\small$\eta_5$}}
\put(104,11){\makebox(0,0)[cc]{\small$\eta_6$}}
\put(150,11){\makebox(0,0)[cc]{\small$\eta_7=\gamma_2$}}
\put(197,11){\makebox(0,0)[cc]{\small$\eta_8$}}
\put(282,11){\makebox(0,0)[cc]{\small$\mu=\eta_9=\gamma_3=\tau_3$}}
\put(75,-70){\makebox(0,0)[cc]{\small$\eta_{11}$}}
\put(75,33){\makebox(0,0)[cc]{\small$\eta_{13}$}}
\put(187,-60){\makebox(0,0)[cc]{\small$\eta_{12}$}}
\put(225,33){\makebox(0,0)[cc]{\small$\eta_{14}$}}
\put(285,-11){\makebox(0,0)[cc]{\small$\eta_{10}$}}
{\color{lightgray}
\thinlines\thinlines
\put(33,-13){\vector(4,3){10}}
\qbezier(8,-3)(35,-4)(45,-35)
\begin{rotate}{-40}
\put(32,4.5){\makebox(0,0)[cc]{\tiny$(\gamma_0,\tau_1)$}}
\end{rotate}
\put(113,-27){\vector(4,3){15}}
\thinlines
\qbezier(58,-3)(120,-10)(145,-75)
\begin{rotate}{-42}
\put(101,50){\makebox(0,0)[cc]{\tiny$(\gamma_1,\tau_2)$}}
\end{rotate}
\put(208,-14){\vector(3,1){15}}
\thinlines
\qbezier(158,-5)(230,-25)(246,-5)
\begin{rotate}{-10}
\put(196,16){\makebox(0,0)[cc]{\tiny$(\gamma_2,\tau_3)$}}
\end{rotate}
}
\end{picture}

\noindent 
Here the open circles represent free points. We now have $\eta_1 \subset \cdots \subset \eta_{10}$. Moreover, 
$\eta_2\subset \eta_{11}$, $\eta_5\subset \eta_{12}$, $\eta_6\subset \eta_{13}$ and $\eta_9\subset \eta_{14}$. Since we are interested in the vertices 
to which $\mu$ is infinitely near to, we may concentrate on the chain $\eta_1 \subset \cdots \subset \eta_9$ or, in 
the dual graph, blow down the vertices $\eta_i$ with $i > 9$. The dashed lines in the graph 
represent the edges emerging when blowing down. Obviously, the maximal free point to 
which $\mu=\eta_9$ is infinitely near to is $\mu$ itself, and further, the maximal non free point to 
which $\mu$ is infinitely near to is $\eta_7$. Thus the pair $(\eta_7,\mu)$ is associated to $\mu$. Similarly, the 
pair $(\eta_3,\eta_4)$ is associated to $\eta_7$ and $(\eta_1,\eta_2)$ is associated to $\eta_3$. 
\end{rem}

\subsubsection*{Jumping numbers}

We will next recall the definition of jumping numbers. A general reference for jumping 
numbers is the fundamental article~\cite{ELSV}. For a nonnegative rational number $\xi$, the 
\textit{multiplier ideal} $\cJ(\fa^\xi)$ is defined to be the ideal 
$$
\cJ(\fa^\xi):=\Gamma\left(X,\cO_X\left(K-\left\lfloor\xi D\right\rfloor\right)\right)\subset R,
$$
where $D = d_1E_1 +\cdots+ d_NE_N$ is the divisor corresponding to $\fa$ and $\left\lfloor \xi D\right\rfloor$ denotes the 
integer part of $\xi D$. It is now known that there is an increasing discrete sequence 
$$
0=\xi_0<\xi_1<\xi_2<\cdots
$$ 
of rational numbers $\xi_i$ characterized by the properties that $\cJ(\fa^\xi)=\cJ(\fa^{\xi_i})$ for $\xi\in[\xi_i,\xi_{i+1})$, 
while $\cJ(\fa^{\xi_{i+1}})\subsetneq\cJ(\fa^{\xi_i})$ for every $i$. The numbers $\xi_1,\xi_2,\dots$, are called the 
\textit{jumping numbers} of $\fa$. The following Proposition~\ref{parametrisointi}, which is fundamental for the rest of 
this article, results from~\cite[Proposition 6.7 and Proposition 7.2]{J}.

\begin{prop}\label{parametrisointi}
Let $\fa \subset R$ be a complete ideal of finite colength. Then $\xi$ is a jumping 
number of $\fa$ if and only if there exists an antinef divisor $F=fE\in \Lambda$ such that 
$$
\xi=\xi_F:=\min_\nu \frac{f_\nu+k_\nu+1}{d_\nu}.
$$ 
Moreover, if $\fb$ is the complete ideal corresponding to $F$, then 
$$
\xi=\inf\{c\in\mathbb Q_{>0}\mid\cJ(\fa^{c})\nsupseteq\fb\}.
$$
\end{prop}

\begin{notation}
We write for any two divisors $F=fE,G=gE\in\Lambda_\mathbb Q$ and for any vertex $\nu$ 
$$
\lambda(F,G;\nu):=\frac{f_\nu+k_\nu+1}{g_\nu}.
$$
For any integer $a$ we set
$$
\lambda(a,\nu)=\lambda(a,D;\nu):=\lambda(aE,D;\nu).
$$ 
Furthermore, we call the set 
$$
\{\nu\in\Gamma\mid\lambda(f_\nu,\nu)=\xi\}
$$ 
the \textit{support} of the jumping number $\xi$ with respect to the divisor $F$. The set of jumping 
numbers of $\fa$ supported at a vertex $\mu\in\Gamma$ is denoted by
$$
\mathcal H_\mu^\fa:=\{\xi_F\mid F\in\Lambda\textit{ is antinef and }\xi_F=\lambda(F,D;\mu)\}.
$$
\end{notation}

Recall that the function $\lambda_F:|\Gamma|\to\mathbb Q$, where $F=\sum_{\nu\in\Gamma}f_\nu E_\nu$ is a divisor and 
$\lambda_F(\nu)=\lambda(f_\nu,\nu)$, makes the dual graph as an ordered tree. In \cite{HJ} we investigated this 
kind of ordered tree structures, and further, we proved that a number being a jumping 
number is equivalent to the existence of certain kind of ordered tree structures. In the 
sequel, we make use of these results. 

\begin{rem}\label{DualG}
Note that in \cite{HJ} and \cite{J} $\Gamma$ is the dual graph of the minimal principalization 
of $\fa$. We may loosen this restriction and consider the dual graph of a principalization 
of the ideal. In the sequel, we may think $\Gamma$ as a dual graph of any ideal corresponding 
to some antinef divisor in $\Lambda$. This is convenient, and it is possible because if $\fb$ is such 
an ideal, then the principalization corresponding to $\Gamma$ is a principalization of $\fb$, and the 
minimal principalization is obtained by blowing down. Observe that the ordered tree 
structures behave accordingly. Suppose that the divisor corresponding to $\fb$ is $gE$ and 
that the dual graph $\Gamma_\fb$ of its minimal principalization is obtained by blowing down a 
vertex $\nu\in\Gamma$, then the valuation matrix of $\fb$ is just a restriction of that of $\fa$. For a divisor 
$fE\in\Lambda_\mathbb Q$ and for a vertex $\gamma\in\Gamma_\fb$ we get $\lambda(fE|_{\Gamma_\fb},gE|_{\Gamma_\fb};\gamma)=\lambda(fE,gE;\gamma)$. Thus the 
ordered tree structures provided by $\lambda$ (see \cite{HJ}) can be obtained as restrictions, as well.
\end{rem}

Recall our main result in \cite[Theorem 1]{HJ}:
 
\begin{thm}[Theorem 1 in \cite{HJ}]
We have $\xi\in\mathcal H^\fa_\mu$ if and only if there is a (connected) set 
$U\subset\Gamma$ containing $\mu$ and a set of nonnegative integers $$\{a_{\eta}\in\mathbb N\mid d(\eta,U)\le1\}$$ satisfying
\begin{itemize}
	\item [i)] $\lambda(a_\eta,\eta)>\xi=\lambda(a_\gamma,\gamma)$ for every $\gamma\in U$ and $\eta\sim U$;
	\item [ii)] $w_{\Gamma}(\gamma)a_\gamma\ge\sum_{\nu\sim\gamma}a_\nu$ for every $\gamma\in U$.
\end{itemize}
\end{thm}

\noindent For further use, we also give here a refined versions of \cite[Lemma 5]{HJ} and \cite[Lemma 7]{HJ}.

\begin{lem}
\label{hj5}
Given any vertex $\gamma\in\Gamma$ and any nonnegative integer $a_\gamma$, we may choose for
every vertex $\eta\sim\gamma$ a nonnegative integer $a_\eta$ so that
$$
w_{\Gamma}(\gamma)a_{\gamma}\ge\sum_{\eta\sim\gamma}a_\eta
	\text{ and }
\lambda(a_\eta,\eta)\ge\lambda(a_\gamma,\gamma),
$$
where the latter inequality holds for each $\eta\sim\gamma$ except at most one. More precisely, if
$$
\{\eta\mid\eta\sim\gamma\}=\{\eta_1,\dots\eta_m\},
$$
where $m>1$, then the following is true:
\begin{itemize}
	\item[1)]
		If it is possible to find a nonnegative integer $a_{\eta_1}$ with $\lambda(a_{\eta_1},\eta_1) = \lambda(a_\gamma,\gamma)$, then 
		one may choose the other integers $a_{\eta_j}$ so that
		$$
		\lambda(a_{\eta_2},\eta_2)\ge\lambda(a_\gamma,\gamma)
			\text{ and }
		\lambda(a_{\eta_j},\eta_j)>\lambda(a_\gamma,\gamma)
		$$
    for all $2< j\le m$. 
	\item[2)] 
	If it is possible to find a nonnegative integer $a_{\eta_1}$ satisfying $\lambda(a_{\eta_1},\eta_1)<\lambda(a_\gamma,\gamma)$ or, 
	in the case $\widehat d_\gamma>0$, 	$\lambda(a_{\eta_1},\eta_1)=\lambda(a_\gamma,\gamma)$, then one can choose the other integers $a_\eta$ 
	in such a way that
	$$
	\lambda(a_{\eta_j},\eta_j)>\lambda(a_\gamma,\gamma)
	$$
	holds for every $1<j\le m$.
\end{itemize}
\end{lem}

\begin{proof}
The proof is conducted in \cite[Lemma 5]{HJ} except for the amendment in 1), which 
claims that if we have nonnegative integers $a_{\eta_1}$ and $a_{\eta_2}$ satisfying
\begin{equation}\label{>=>}
	\lambda(a_{\eta_2},\eta_2)
	>\lambda(a_\gamma,\gamma)
	=\lambda(a_{\eta_1},\eta_1)
	>\lambda(a_{\eta_2}-1,\eta_2),
\end{equation}
then we may find nonnegative integers $a_{\eta_j}$ for $2<j\le m$ so that
$$
w_{\Gamma}(\gamma)a_{\gamma}=\sum_{j=1}^m a_{\eta_j}
	\text{ and }
\lambda(a_{\eta_j},\eta_j)\ge\lambda(a_\gamma,\gamma),
$$
where the inequality is strict for $1<j\le m$. 

To prove that, suppose that Equation \eqref{>=>} holds. For $\mu,\nu\in\Gamma$, write 
$$
\ga_{\mu,\nu}:=k_\nu+1-\frac{d_\nu}{d_\mu}(k_\mu+1).
$$ 
Note that if $a_{\eta_2}=0$, then by \cite[Lemma 3 a)]{HJ}
$$
a_\gamma-1\le a_\gamma+\ga_{\eta_2,\gamma}=d_\gamma(\lambda(a_{\gamma},\gamma)-\lambda(a_{\eta_2},\eta_2)),
$$ 
which must be negative. Therefore $a_\gamma=0$, and then similarly, by \cite[Lemma 3 a)]{HJ},
$$
a_{\eta_1}-1
< a_{\eta_1}+\ga_{\gamma,\eta_1}
=d_{\eta_1}(\lambda(a_{\eta_1},\eta_1)-\lambda(a_{\gamma},\gamma))
=0,
$$
so that $a_{\eta_1}=0$, but then the claim follows from \cite[Lemma 3 b)]{HJ}.

Assume then that $a_{\eta_2}>0$ and set $a'_{\eta_2}:=a_{\eta_2}-1$. Now $\lambda(a'_{\eta_2},\eta_2)<\lambda(a_\gamma,\gamma)$, and by 
\cite[Lemma 5]{HJ} we may find nonnegative integers $a'_{\eta_j}$ for $1\le j\le m$, $j\neq2$, so that 
$$
w_{\Gamma}(\gamma)a_{\gamma}=\sum_{j=1}^m a'_{\eta_j}\text{ and }\lambda(a'_{\eta_j},\eta_j)>\lambda(a_\gamma,\gamma)\text{ for }j\neq2.
$$
But then $\lambda(a'_{\eta_1},\eta_1)>\lambda(a_{\eta_1},\eta_1)$. Clearly, we may choose the integers $a'_{\eta_j}$ so that $a'_{\eta_1}=a_{\eta_1}+1$. 
It follows that 
$$
w_{\Gamma}(\gamma)a_{\gamma}=a_{\eta_1}+a_{\eta_2}+\sum_{j=3}^m a'_{\eta_j}
	\text{ and }
\lambda(a'_{\eta_j},\eta_j)>\lambda(a_\gamma,\gamma)
	\text{ for }
2<j\le m.
$$
Choosing now $a_{\eta_j}=a'_{\eta_j}$ for $2<j\le m$ yields the claim.
\end{proof}

Practically, the lemma shows that if $\lambda_F(\mu)$ is a local minimum for a function $\lambda_F$ 
where $F$ is an effective divisor, then we may find an antinef divisor $A=\sum_\nu a_\nu E_\nu$ for 
which $\lambda(a_\mu,\mu)=\lambda(f_\mu,\mu)$ is the global minimum of the function $\lambda_A$. The only problem 
that may arise in finding such integers $a_\nu$ is the situation where we already have integers 
$a_\gamma$ and $a_\tau$ with $\lambda(a_\gamma,\gamma)=\lambda(a_\tau,\tau)$, and $\tau$ is an end vertex. We cannot go on choosing 
integers $a_\nu$ with $\tau\sim\nu\neq\gamma$ and $\lambda(a_\nu,\nu)=\lambda(a_\tau,\tau)$, because there are no such vertices, 
and it may happen that $\widehat a_\tau:=w_\gamma(\tau)a_\tau-a_\gamma<0$. This is the reason why we cannot apply 
1) of Lemma \ref{hj5} to an end. Nevertheless, rephrasing \cite[Lemma 7]{HJ}, the next Lemma shows 
that 2) of Lemma \ref{hj5} is applicable even if the vertex in question is an end. 

\begin{lem}\label{hj7}
Suppose $\tau$ is an end and $\gamma$ is adjacent to it. If $a_\tau$ and $a_\gamma$ are such integers 
that $\lambda(a_\gamma,\gamma)\le\lambda(a_\tau,\tau)$, where the equality holds only if $\widehat d_\tau>0$, then
$$
w_\Gamma(\tau)a_\tau\ge a_\gamma.
$$
\end{lem}

\begin{proof}
By Equation \eqref{K} we know that $\widehat k_\tau=2-w_\Gamma(\tau)$. On the other hand, by Equation~\eqref{P} 
we have $\widehat k_\tau=w_\Gamma(\tau)k_\tau-k_\gamma$. Thus $w_\Gamma(\tau)(k_\tau+1)=k_\gamma+2$. Moreover, by Equation~\eqref{P} 
we get $w_\Gamma(\tau)d_\tau=d_\gamma+\widehat d_\tau$. This shows that
$$
\lambda(a_\tau,\tau)
=\frac{w_\Gamma(\tau)(a_\tau+k_\tau+1)}{w_\Gamma(\tau)d_\tau}
=\frac{w_\Gamma(\tau)a_\tau+k_\gamma+2}{d_\gamma+\widehat d_\tau}
$$
Therefore we see that
$$
\lambda(a_\gamma,\gamma)
=\frac{a_\gamma+k_\gamma+1}{d_\gamma}
<\frac{(w_\Gamma(\tau)a_\tau+1)+k_\gamma+1}{d_\gamma},
$$
which implies that $a_\gamma<w_\Gamma(\tau)a_\tau+1$, as wanted.
\end{proof}

By using these results we may construct suitable ordered tree structures, which in 
turn can prove that certain rationals are jumping numbers for our ideal supported at the 
desired vertex or vertices. The next lemma shows that in order to determine the jumping 
numbers of an ideal, we just need to know the jumping numbers supported at a vertex 
which is either a star or corresponds to a simple factor of the ideal.

\begin{lem}\label{3d}
A support of a jumping number contains a vertex which is either a star or 
corresponds to a simple factor of the ideal.
\end{lem} 

\begin{proof}
Let $\Gamma$ be a dual graph of an ideal $\fa=\prod_{\nu\in\Gamma}\fp_\nu^{\widehat d_\nu}$. Suppose $\xi$ is a jumping number 
of $\fa$ supported at a vertex $\gamma\in\Gamma$, for which $\widehat d_\gamma=0$ and $v_\Gamma(\gamma)<3$. By Proposition \ref{parametrisointi} we 
have an antinef divisor $F$ for which $\xi=\xi_F$. Further, we have $\xi=\lambda(f_\gamma,\gamma)\le\lambda(f_\nu,\nu)$ for 
any $\nu\in\Gamma$. As $\widehat k_\gamma=2-w_\Gamma(\gamma)$ by \eqref{K}, and on the other hand, $\widehat k_\gamma=w_\Gamma(\gamma)k_\gamma-\sum_{\eta\sim\gamma}k_\eta$, 
we see that $w_\Gamma(\gamma)(k_\gamma+1)=2+\sum_{\eta\sim\gamma}k_\eta$. By this and by \eqref{P} we obtain
$$
\lambda(f_\gamma,\gamma)
=\frac{w_\Gamma(\gamma)f_\gamma+w_\Gamma(\gamma)(k_\gamma+1)}{w_\Gamma(\gamma)d_\gamma}
=\frac{\widehat f_\gamma+2-v_\Gamma(\gamma)+\sum_{\eta\sim\gamma}(f_\eta+k_\eta+1)}{\widehat d_\gamma+\sum_{\eta\sim\gamma}d_\eta}.
$$ 
Since $\widehat d_\gamma=0$, $v_{\Gamma}(\gamma)<3$ and $\lambda(f_\eta,\eta)\ge\lambda(f_\gamma,\gamma)$, the above yields $v_\Gamma(\gamma)=2$, $\widehat f_\gamma=0$ and 
$\lambda(f_\eta,\eta)=\xi$ for any $\eta\sim\gamma$. In other words, $\gamma$ has exactly two adjacent vertices, which 
both are in the support of $\xi$. If neither of them is a star nor corresponds to a factor, we 
may apply the above to them. Because the dual graph contains finitely many vertices, 
we must eventually come up with a vertex in the support of $\xi$, which is either a star or 
corresponds to a factor.
\end{proof}

\section{Modifications of the factorization}

Let $\fa$, $D$ and $\Gamma$ be as above and let $V$ be the valuation matrix, $\widehat d\in\mathbb Q^\Gamma$ the factorization 
vector and $d = \widehat d V$ the valuation vector of $\fa$. According to \cite[Theorem 1 and Lemma~6]{HJ}, 
we may find for any antinef divisor $F=fE\in\Lambda$ an antinef divisor $G=gE$ such that 
$\widehat g_\nu>0$ only if $\nu$ is an end of $\Gamma$ and $\xi_F=\xi_{G}$. In this paper we further investigate divisors 
corresponding to a jumping number and develop a method to modify them in order to 
find ideals sharing the jumping numbers supported at a given vertex. To begin with, let 
us consider the mapping $\rho_{[\mu,\gamma]}:\Gamma\to\mathbb Q$, where 
$$
\rho_{[\mu,\gamma]}:\nu\mapsto\frac{V_{\gamma,\nu}}{V_{\mu,\nu}}.
$$

\begin{lem}\label{V<V}
The mapping $\rho_{[\mu,\gamma]}$ is strictly increasing on the path going from $\mu$ to $\gamma$ and it 
stays constant on any path going away from $[\mu,\gamma]$, in other words, $\rho_{[\mu,\gamma]}(\nu_1)<\rho_{[\mu,\gamma]}(\nu_2)$ 
if and only if $[\mu,\nu_1]\cap[\mu,\gamma]\subsetneq[\mu,\nu_2]\cap[\mu,\gamma]$. Moreover, if $\mu\sim\gamma$, then
\begin{equation}\label{Oletus}
\rho_{[\mu,\gamma]}(\gamma)=\frac{V_{\gamma,\mu}+1}{V_{\mu,\mu}}.
\end{equation}
\end{lem}

\begin{proof}
Note first that since $\rho_{[\gamma,\mu]}(\nu)\rho_{[\mu,\gamma]}(\nu)=1$ for every $\nu$, the claim holds for $\rho_{[\gamma,\mu]}$ 
exactly when it holds for $\rho_{[\mu,\gamma]}$. If the claim holds whenever $\mu$ and $\gamma$ are adjacent vertices, 
then we get the desired result by induction on the distance of $\mu$ and $\gamma$, as 
$$
\rho_{[\mu,\gamma]}(\nu)=\rho_{[\mu,\eta]}(\nu)\rho_{[\eta,\gamma]}(\nu).
$$
Hence, it is enough to consider the cases where $\mu\sim\gamma$. We proceed by using induction 
on the number of vertices of the dual graph.

If $\Gamma$ consists of only two adjacent vertices, then 
$$
V=\hspace{-3pt}\left[\begin{array}{rr}1&1\\1&2\end{array}\right]
$$
and the case is clear.

Suppose that $\Gamma=\Gamma'(\eta, U)$ and that the claim holds on the graph $\Gamma'$. Note that the 
valuation matrix of $\Gamma'$ is just a restriction of that of $\Gamma$. Moreover, since the valuation 
matrix of $\Gamma$ is $V^{\textsc t}=V=(P^\textsc tP)^{-1}$, we see that $P_\eta V_\gamma=q_{\gamma,\eta}$, i.~e.,
\begin{equation}\label{Vi}
V_{\gamma,\eta}=\sum_{i\prec\eta}V_{\gamma,i}+\delta_{\eta,\gamma}.
\end{equation}
If now $\eta\notin[\mu,\gamma]$, then $\rho_{[\mu,\gamma]}(\nu)$ remains unaltered when $\nu\neq\eta$, and if $\nu=\eta$ then $j\in\Gamma^{\mu}_\gamma$ 
for any $j\prec\eta$ exactly when $\eta\in\Gamma_\gamma^\mu$. Therefore for any $j\prec\eta$,
$$
\rho_{[\mu,\gamma]}(\eta)
=\frac{\sum_{i\prec\eta}V_{\gamma,i}}{\sum_{i\prec\eta}V_{\mu,i}}
=\rho_{[\mu,\gamma]}(j).
$$

It remains to show that if $\mu\sim\eta$, then the claim holds for $\rho_{[\mu,\eta]}$, too. If  $U=\{\mu\}$ and 
$\nu\neq\eta$, then by Equation \eqref{Vi} we see that $V_{\eta,\nu}=V_{\mu,\nu}$, and further,
$$
\rho_{[\mu,\eta]}(\nu)
=\frac{V_{\eta,\mu}}{V_{\mu,\mu}}
<\frac{V_{\eta,\mu}+1}{V_{\mu,\mu}}
=\frac{V_{\eta,\eta}}{V_{\mu,\eta}}
=\rho_{[\mu,\eta]}(\eta),
$$
as wanted. Especially, Equation \eqref{Oletus} holds in this case. 

Suppose then that $U=\{\mu,\gamma\}$. Together with \eqref{Oletus} Equation \eqref{Vi} yields
\begin{align*}
	\rho_{[\mu,\eta]}(\gamma)
	&=\frac{V_{\mu,\gamma}+V_{\gamma,\gamma}}{V_{\mu,\gamma}}=1+\frac{V_{\gamma,\mu}+1}{V_{\mu,\mu}}\\
	&=1+\frac{V_{\gamma,\gamma}+V_{\gamma,\mu}+1}{V_{\mu,\gamma}+V_{\mu,\mu}}
		=\frac{V_{\mu,\eta}+V_{\gamma,\eta}+1}{V_{\mu,\eta}}
		=\rho_{[\mu,\eta]}(\eta).
\end{align*}
Moreover, 
$$
\rho_{[\mu,\eta]}(\eta)
=\rho_{[\mu,\eta]}(\gamma)
=\frac{V_{\eta,\gamma}}{V_{\mu,\gamma}}=\frac{V_{\eta,\mu}+V_{\eta,\gamma}+1}{V_{\mu,\mu}+V_{\mu,\gamma}}
=\frac{V_{\eta,\mu}+1}{V_{\mu,\mu}}
>\rho_{[\mu,\eta]}(\mu).
$$
This shows that Equation \eqref{Oletus} holds for $\eta$. Since 
$$
\rho_{[\mu,\eta]}(\nu)
=\frac{V_{\mu,\nu}+V_{\gamma,\nu}}{V_{\mu,\nu}}
=1+\rho_{[\mu,\gamma]}(\nu)
$$
for every $\nu\neq\eta$, we see that $\rho_{[\mu,\eta]}$ stays constant on any path going away from $[\mu,\eta]$.
Hence the claim holds for $\rho_{[\mu,\eta]}$, too. 
\end{proof}

\begin{prop}\label{D}
Write $\mathbf 1_i=(\delta_{i,j})_{j\in\Gamma}$. For any vertices $\gamma$, $\mu$ and $\eta$, set
$$
\widehat r_{[\mu,\gamma]}
:=\mathbf 1_\gamma-\rho_{[\mu,\gamma]}(\mu)\mathbf 1_\mu
	\:\:\:\text{ and }\:\:\:
\varphi_\eta(\nu)=\varphi_\eta^{[\mu,\gamma]}(\nu)
:=\frac{\left(\widehat r_{[\mu,\gamma]}V\right)_\nu}{V_{\eta,\nu}}.
$$
Then $\varphi_\eta(\nu)\ge0$, where the inequality is strict if and only if $\nu\in\Gamma^\mu_\gamma$. If $\nu,\nu'\in[\mu,\gamma]$ and 
$d(\mu,\nu)<d(\mu,\nu')$, then
$$
\varphi_\eta(\nu)<\varphi_\eta(\nu').
$$
Further, if $\eta\in[\mu,\gamma]$ then $\varphi_\eta(\nu)$ is constant on any path intersecting $[\mu,\gamma]$ at most on 
one point.
\end{prop}

\begin{proof}
We have 
\begin{equation*}
\varphi_\eta(\nu)
=\rho_{[\eta,\gamma]}(\nu)-\rho_{[\mu,\gamma]}(\mu)\rho_{[\eta,\mu]}(\nu)
=\left(\rho_{[\mu,\gamma]}(\nu)-\rho_{[\mu,\gamma]}(\mu)\right)\rho_{[\eta,\mu]}(\nu)
\end{equation*}
By Lemma \ref{V<V} $\rho_{[\mu,\gamma]}(\nu)\ge\rho_{[\mu,\gamma]}(\mu)$ and thereby also $\varphi_\eta(\nu)\ge0$, where the equality holds 
exactly when $\nu\in\Gamma\smallsetminus\Gamma^\mu_\gamma$.

Suppose $\nu,\nu'\in[\mu,\gamma]$ and $d(\mu,\nu)<d(\mu,\nu')$. If $]\mu,\eta]\cap[\mu,\gamma]=\emptyset$, then $\rho_{[\eta,\mu]}(\nu)$ does 
not depend on $\nu\in[\mu,\gamma]$, and again by Lemma \ref{V<V} we know that $\rho_{[\mu,\gamma]}(\nu)$ is strictly 
increasing on the path going from $\mu$ to $\gamma$, which proves the case. If $]\mu,\eta]\cap[\mu,\gamma]\neq\emptyset$, 
then $\rho_{[\eta,\gamma]}(\nu)$ is strictly increasing on $[\eta,\gamma]$ and $\rho_{[\eta,\mu]}(\nu)$ is strictly decreasing on $[\mu,\eta]$, 
while $\rho_{[\mu,\gamma]}(\mu)$ is a constant. Therefore
$$
\varphi_\eta(\nu)
=\rho_{[\eta,\gamma]}(\nu)-\rho_{[\mu,\gamma]}(\mu)\rho_{[\eta,\mu]}(\nu)
<\rho_{[\eta,\gamma]}(\nu')-\rho_{[\mu,\gamma]}(\mu)\rho_{[\eta,\mu]}(\nu')
=\varphi_\eta(\nu').
$$
The rest is now clear.
\end{proof}

In the sequel, we shall make use of the above especially in situations where we have 
a divisor $F$ with $\xi=\xi_F=\lambda(F,D;\mu)$ for a vertex $\mu$ and we want to modify either $F$ or $D$ 
or both in such a way that we still have $\xi=\lambda(F',D';\mu)\le\lambda(F',D';\nu)$ for every $\nu\in\Gamma$. 
For that we introduce a modified factorization vector: Let $\widehat f,\widehat g,\widehat h\in\mathbb Q^\Gamma$ and let $\mu\in\Gamma$. 
We concentrate on $\mu$ and the vertices adjacent to it and modify $\widehat f$ with $\widehat g$ and $\widehat h$ so that 
we 'bring' factors $\widehat g_i$ from each branch emanating from $\mu$ to the closest vertex adjacent 
to $\mu$ and 'distribute' the factor $\sum_{\nu\sim\mu}\rho_{[\mu,\nu]}(\mu)\widehat h_\nu$ from $\mu$ to the adjacent vertices.

\begin{notation}
Let us write $\widehat f^{\,\mu}_{\langle\, \widehat g\,\rangle[\,\widehat h\,]}$ or, if the vertex $\mu$ is clear from the context, just
$$
\widehat f_{\langle\,\widehat g\,\rangle[\,\widehat h\,]}
:=\widehat f-\sum_{i\sim\mu}\sum_{j\in\Gamma_i^\mu}\widehat g_j\widehat r_{[i,j]}+\sum_{i\sim\mu}\widehat h_i\widehat r_{[\mu,i]}.
$$
If either $\widehat g$ or $\widehat h$ is zero, we may omit it in the notation. Let us also set
$$
\widehat f^\mathcal N:=\widehat f-\sum_{i\in\Gamma}\widehat f_i\widehat r_{[\mu,i]}.
$$
\end{notation}

\begin{rem}\label{Fuf}
Suppose $\widehat F=\widehat f_{\langle\,\widehat g\,\rangle[\,\widehat h\,]}$. Then obviously $\widehat f=\widehat F_{\langle\,-\widehat g\,\rangle[\,-\widehat h\,]}$. Moreover, we have 
$$
\widehat f^\mathcal N=\widehat f_{\langle\,\widehat f\,\rangle [\,-\widehat f_{\langle\,\widehat f\,\rangle}\,]}
	\text{\,,\, i.~e.,\, }
\widehat f^\mathcal N_{[\,\widehat f_{\langle\,\widehat f\,\rangle}\,]}=\widehat f_{\langle\,\widehat f\,\rangle}.
$$
\end{rem}

\begin{lem}\label{dM}
Let $fE=\widehat f\widehat E$ be a divisor. Write $U_i:=\Gamma^\mu_i\smallsetminus\{i\}$. Then
$$
\left(\widehat f_{\langle\, \widehat g\,\rangle[\,\widehat h\,]}\right)_i=
\begin{cases}\displaystyle
\widehat f_\mu-\sum_{\nu\sim\mu}\rho_{[\mu,\nu]}(\mu)\widehat h_\nu&\text{ if }i=\mu\\
\displaystyle\widehat f_i+\widehat h_i+\sum_{j\in U_i}\rho_{[i,j]}(\mu)\widehat g_j&\text{ if }i\sim\mu\\
\widehat f_i-\widehat g_i&\text{ otherwise. }
\end{cases}
$$
Furthermore,
$$
(\widehat f_{\langle\,\widehat g\,\rangle}V)_i=f_i-\sum_{\nu\sim\mu}\sum_{j\in\Gamma_\nu^\mu}\widehat g_j\varphi_\mu^{[\nu,j]}(i)V_{\mu,i}.
$$
It follows that if $\widehat g\ge0$, then $(\widehat f_{\langle\,\widehat g\,\rangle}V)_i\le f_i$, where the equality holds when $d(\mu,i)\le1$, 
or more precisely, the inequality is strict exactly when $i\in\Gamma_j^\nu$ for some $\nu\sim\mu$ and some 
$j\in\Gamma_\nu^\mu$ with $\widehat g_j>0$. 

Similarly, 
$$
(\widehat f_{[\,\widehat h\,]}V)_i=f_i+\sum_{\nu\sim\mu}\widehat h_\nu\varphi_\mu^{[\mu,\nu]}(i)V_{\mu,i},
$$
and if $\widehat h\ge0$, then $(\widehat f_{[\,\widehat h\,]}V)_i\ge f_i$, where the inequality is strict exactly when $i\in\Gamma_\nu^\mu$ for 
such $\nu\sim\mu$ that $\widehat h_\nu>0$.
\end{lem}

\begin{proof}
Recall that $\widehat r_{[i,i]}=0$ for any $i$. Straightforward calculation shows that
\begin{align*}
\widehat f_{\langle\,\widehat g\,\rangle[\,\widehat h\,]}
&=\widehat f-\sum_{i\sim\mu}\sum_{j\in U_i}\widehat g_j(\mathbf 1_j-\rho_{[i,j]}(i)\mathbf 1_i)
	+\sum_{i\sim\mu}\widehat h_i(\mathbf 1_i-\rho_{[\mu,i]}(\mu)\mathbf 1_\mu)\\
&=\widehat f-\sum_{i\sim\mu}\rho_{[\mu,i]}(\mu)\widehat h_i\mathbf 1_\mu
	+\sum_{i\sim\mu}\left(\widehat h_i+\sum_{j\in U_i}\rho_{[i,j]}(\mu)\widehat g_j\right)\mathbf 1_i
	-\sum_{i\sim\mu}\sum_{j\in U_i}\widehat g_j\mathbf 1_j
\end{align*}
as $\rho_{[i,j]}(i)=\rho_{[i,j]}(\mu)$ by Lemma \ref{V<V}. This proves the first assertion. Furthermore, by 
Proposition \ref{D} we observe that
$$
(\widehat f_{\langle\,\widehat g\,\rangle}V)_i
=f_i-\sum_{\nu\sim\mu}\sum_{j\in\Gamma_\nu^\mu}\widehat g_j(\widehat r_{[\nu,j]}V)_i
=f_i-\sum_{\nu\sim\mu}\sum_{j\in\Gamma_\nu^\mu}\widehat g_j\varphi_\mu^{[\nu,j]}(i)V_{\mu,i},
$$
where $\varphi_\mu^{[\nu,j]}(i)\ge0$ is positive if and only if $i\in\Gamma_j^\nu$. Assuming $\widehat g\ge0$, this shows that 
$$
\sum_{\nu\sim\mu}\sum_{j\in\Gamma_\nu^\mu}\widehat g_j\varphi_\mu^{[\nu,j]}(i)V_{\mu,i}>0
$$ 
if and only if $i\in\Gamma_j^\nu$ for some $\nu\sim\mu$ and some $j\in\Gamma_\nu^\mu$ with $\widehat g_j>0$.

Similarly, by Proposition \ref{D} we get
$$
(\widehat f_{[\,\widehat h\,]}V)_i
=f_i+\sum_{\nu\sim\mu}\widehat h_\nu(\widehat r_{[\mu,\nu]}V)_i
=f_i+\sum_{\nu\sim\mu}\widehat h_\nu\varphi_\mu^{[\mu,\nu]}(i)V_{\mu,i},
$$
where $\varphi_\mu^{[\mu,\nu]}(i)>0$ exactly when $i\in\Gamma^\mu_\nu$. Thus, if $\widehat h\ge0$,
$$
\sum_{\nu\sim\mu}\widehat h_\nu\varphi_\mu^{[\mu,\nu]}(i)V_{\mu,i}>0
$$
exactly when $i\in\Gamma^\mu_\nu$ for such $\nu\sim\mu$ that $\widehat h_\nu>0$.
\end{proof}

\begin{lem}\label{dM0}
For any divisor $\widehat f\widehat E$ we have
$$
\widehat f^\mathcal N=\sum_{i\in\Gamma}\rho_{[\mu,i]}(\mu)\widehat f_i\mathbf 1_\mu
	\text{\: and \:}
\left(\widehat f^\mathcal NV\right)_\mu=\left(\widehat fV\right)_\mu.
$$
Furthermore, if for some divisor $\widehat g\widehat E$ holds $\lambda(\widehat f\widehat E,D;\mu)=\lambda(\widehat g\widehat E,D;\mu)$, then
$$
\widehat f^\mathcal N=\widehat g^\mathcal N.
$$
\end{lem}

\begin{proof}
The first equality comes straightforwardly from the definition. A direct calculation 
shows that
$$
\left(\widehat f^\mathcal NV\right)_\mu
=\sum_{i\in\Gamma}\rho_{[\mu,i]}(\mu)\widehat f_i V_{\mu,\mu}
=\sum_{i\in\Gamma}\frac{V_{i,\mu}}{V_{\mu,\mu}}\widehat f_i V_{\mu,\mu}
=\sum_{i\in\Gamma}\widehat f_i V_{i,\mu}
=\left(\widehat fV\right)_\mu,
$$
as wanted.

Suppose next that $\lambda(\widehat f\widehat E,D;\mu)=\lambda(\widehat g\widehat E,D;\mu)$. Then $(\widehat f V)_\mu=(\widehat g V)_\mu$. By the above 
we have $(\widehat f^\mathcal NV)_\mu=(\widehat g^\mathcal NV)_\mu$, and further, $\widehat f^\mathcal N_\mu V_{\mu,\mu}=\widehat g^\mathcal N_\mu V_{\mu,\mu}$. This is to say that $\widehat f^\mathcal N_\mu=\widehat g^\mathcal N_\mu$, 
but then $\widehat f^\mathcal N=\widehat g^\mathcal N.$
\end{proof}

\begin{lem}\label{dM1}
For antinef divisors $fE\neq0$ and $G$ and for any vertex $\nu\in\Gamma$
$$
\lambda(G, \widehat f_{\langle \,\widehat f\,\rangle}\widehat E;\nu)\ge\lambda(G,fE;\nu),
$$
where the equality holds exactly when either $d(\mu,\nu)\le1$ or when $\widehat f_j=0$ for every $j\in\Gamma_\nu^i$, 
where $i$ is the vertex in $[\mu,\nu[$ adjacent to $\mu$.
\end{lem}

\begin{proof}
By Lemma \ref{dM} we know that $\widehat f_{\langle \,\widehat f\,\rangle}V_\nu\le f_\nu$, where the equality holds exactly when 
either $d(\mu,\nu)\le1$ or when $\widehat f_j=0$ for every $j\in\Gamma_\nu^i$, where $i$ is the vertex in $[\mu,\nu[$ adjacent 
to $\mu$. Thus the claim is clear.
\end{proof}

\begin{lem}\label{dM2}
Let $F=fE$ be an antinef divisor and suppose $0\neq\widehat d\in\mathbb Q^\Gamma_{\ge0}$. If
$$
\min_{\nu}\lambda(F,\widehat d_{\langle\widehat d\,\rangle}\widehat E;\nu)=\lambda(F,\widehat d_{\langle\widehat d\,\rangle}\widehat E;\mu),
$$
then we can find an antinef divisor $G$ satisfying
$$
\min_{\nu}\lambda(G,dE;\nu)=\lambda(G,dE;\mu)=\lambda(F,dE;\mu).
$$
\end{lem}

\begin{proof}
Observe that for any divisors $U,V\in\Lambda_\mathbb Q$ and for any nonzero $n\in\mathbb Q$,
\begin{equation}\label{mUm}
\lambda(U,V;\nu)=n\lambda(U,nV;\nu).	
\end{equation}
Thus it is not a restriction to assume that both $\widehat d$ and $\widehat d_{\langle\widehat d\,\rangle}$ are in 
$\mathbb N^\Gamma$.

We need to show that there is a suitable vector $a\in\mathbb N^\Gamma$, for which the divisor $G=aE$ 
is as wanted. For $\nu$ with $d(\mu,\nu)\le1$ we set $a_\nu=f_\nu$, so that by Lemma \ref{dM1} we get 
$\lambda(F,\widehat d_{\langle\widehat d\,\rangle}\widehat E;\nu)=\lambda(a_\nu,dE;\nu)$. It follows that $\lambda(a_\nu,dE;\nu)\ge\lambda(a_\mu,dE;\mu)$ and 
$$
\widehat a_\mu:=w_\Gamma(\mu) a_\mu-\sum_{\nu\sim\mu}a_\nu=\widehat f_\mu\ge0.
$$

For any branch $\Gamma_\gamma^\mu$ with $\gamma\sim\mu$ we have two possible cases: either $\widehat d_\eta=0$ for every $\eta\in\Gamma_\gamma^\mu$, or 
$\widehat d_\eta>0$ for some $\eta\in\Gamma_\gamma^\mu$. In the first case, we see by Lemma \ref{dM1} that 
$\lambda(F,\widehat d_{\langle\widehat d\,\rangle}\widehat E;\nu)=\lambda(F,dE;\nu)$ for $\nu\in\Gamma_\gamma^\mu$. Hence we may choose $a_\nu=f_\nu$ for every $\nu\in\Gamma_\gamma^\mu$, 
so that $\lambda(a_\nu,dE;\nu)\ge\lambda(a_\mu,dE;\mu)$ and $\widehat a_\nu:=w_\Gamma(\nu) a_\nu-\sum_{i\sim\nu}a_i\ge0$. In the latter case, 
it may happen that $\lambda(F,dE;\nu)<\lambda(F,dE;\mu)$ for some $\nu\in\Gamma_\gamma^\mu\smallsetminus\{\gamma\}$, so that the integers
$f_\nu$ for $\nu\in\Gamma_\gamma^\mu\smallsetminus\{\gamma\}$ won't do. Therefore we must apply Lemma \ref{hj5} in selecting suitable set 
of integers. If $\lambda(F,dE;\gamma)>\lambda(F,dE;\mu)$, then this would be straightforward, since then 
we could by Lemma \ref{hj5} choose integers $a_\nu$ for $\nu\in\Gamma_\gamma^\mu\smallsetminus\{\gamma\}$ so that $\lambda(a_\nu,dE,\nu)$ strictly 
increases on every path in $\Gamma_\gamma^\mu$ going away from $\gamma$, and $\widehat a_\nu\ge0$. Recall that by Lemma~\ref{hj7} 
we can apply Lemma \ref{hj5}, 2) to end vertices, too. In general, we may by using lemma~\ref{hj5} 
find such integers $a_\nu$ for $\nu\in\Gamma_\gamma^\mu\smallsetminus\{\gamma\}$, that $\lambda(a_\nu,dE,\nu)$ is increasing on the path $[\mu,\eta]$, 
and strictly increases on every path in $\Gamma_\gamma^\mu$ going away from $[\mu,\eta]$, and $\widehat a_\nu\ge0$. This can 
be shown as follows.

Let $\eta\in\Gamma_\gamma^\mu$ be such that $\widehat d_\eta>0$, and write $\mu=\eta_0$ and $\gamma=\eta_1$. We have a path of 
adjacent vertices $\eta_0\sim\cdots\sim\eta_k=\eta$ for some positive integer $k$. Since $\lambda(a_{\eta_1},dE;\eta_1)\ge\lambda(a_{\eta_0},dE;\eta_0)$, 
we may by Lemma \ref{hj5} choose integers $a_\nu$ for $\eta_0\neq\nu\sim\eta_1$ so that $\widehat a_{\eta_1}\ge0$ 
and $\lambda(a_\nu,dE;\nu)\ge\lambda(a_{\eta_1},dE;\eta_1)$, where the equality takes place only if $\nu=\eta_2$. Similarly, 
if $0<i\le k$ and $\lambda(a_{\eta_i},dE;\eta_i)\ge\lambda(a_{\eta_{i-1}},dE;\eta_{i-1})$, we may by Lemma \ref{hj5} choose integers 
$a_\nu$ for $\eta_{i-1}\neq\nu\sim\eta_i$ so that $\widehat a_{\eta_i}\ge0$ and $\lambda(a_\nu,dE;\nu)\ge\lambda(a_{\eta_i},dE;\eta_i)$, where the equality
takes place only if $i<k$ and $\nu=\eta_{i+1}$.

If $\theta\sim\eta_i$ for some $i\in\{1,\dots,k\}$ and $\theta\notin[\mu,\eta]$, then we have 
$$
\lambda(a_{\theta},dE;\theta)>\lambda(a_{\eta_i},dE;\eta_i)\ge\lambda(a_{\mu},dE;\mu).
$$ 
Again by using Lemma \ref{hj5} we may choose integers $a_\nu$ for $\nu\in\Gamma_\theta^{\eta_i}$ so that $\widehat a_\nu\ge 0$ and 
$\lambda(a_\nu,dE;\nu)>\lambda(a_{\theta},dE;\theta)$. Subsequently, by applying Lemma \ref{hj5} (and Lemma \ref{hj7}), we may 
find a collection of non-negative integers which meets the requirements of \cite[Theorem~1]{HJ}. 
Thereby we obtain the desired vector.
\end{proof}

\begin{rem}\label{nUn}
By Equation \eqref{mUm} at the beginning of the proof of Lemma \ref{dM2} we see that 
$\xi\in\mathcal H_\mu^{\fa^n}$ if and only if $n\xi\in\mathcal H_\mu^{\fa}$. Thus we may always consider powers $\fa^n$ with $n\in\mathbb N$ big 
enough to achieve the situation where both $\widehat d$ and $\widehat d_{\langle\widehat d\,\rangle}$ are in $\mathbb N^\Gamma$.
\end{rem}

\begin{lem}\label{b-a}
Let $\fa$ be an ideal with a factorization vector $\widehat d$. Suppose $\fa$ is such that 
$\widehat d_{\langle\widehat d\,\rangle}\in\mathbb N^\Gamma$ and let $\fb$ be the ideal corresponding to it. Then $\xi\in\mathcal H_\mu^\fa$ if and only if $\xi\in\mathcal H_\mu^\fb$.
\end{lem}

\begin{proof}
If $\xi\in\mathcal H_\mu^\fa$, then there is an antinef divisor $F$ with $\xi_F=\lambda(F,D;\mu)$. It follows from 
Lemma \ref{dM} that $\xi=\lambda(F,D;\mu)=\lambda(F,\widehat d_{\langle\widehat d\,\rangle}\widehat E;\mu)\le\lambda(F,\widehat d_{\langle\widehat d\,\rangle}\widehat E;\nu)$, as $(\widehat d_{\langle\widehat d\,\rangle} V)_i\le d_i$ where 
the equality holds for $i$ with $d(\mu,i)\le1$. This means that $\xi\in\mathcal H_\mu^\fb$.

If $\xi\in\mathcal H_\mu^\fb$, then $\xi=\lambda(F,\widehat d_{\langle\widehat d\,\rangle};\mu)\le\lambda(F,\widehat d_{\langle\widehat d\,\rangle};\nu)$, then by Lemma \ref{dM2} we have such an 
antinef divisor $G$ that $\xi=\lambda(G,D;\mu)\le\lambda(G,D;\nu)$, which shows that $\xi\in\mathcal H_\mu^\fa$.
\end{proof}

\section{Semigroup of values}

Let $\mathcal{SV}_\mu=(\mathcal{SV}_\mu,+)$ be the submonoid of $\mathbb N$ generated by values $V_{\mu,i}$, $i\in\Gamma$. This is 
called the value semigroup of $v_\mu$. Recall that if $\Gamma$ is the dual graph of the simple ideal 
$\fp_\mu$, then the Zariski exponents are the values of the form $V_{\mu,\tau}$ where $\tau$ is an end (see, e.~g., 
\cite[Remark 6.6]{J}). In general, with any dual graph, we may consider values $V_{\mu,\tau}$ where $\tau$ 
is an end of the graph. We then get the following:
\begin{prop}
Let $\Gamma$ be a dual graph containing $\mu$. As a submonoid of $\mathbb N$, the semigroup 
$\mathcal{SV}_\mu$ is always generated by the set of Zariski exponents of $\mu$, i.~e., the values 
$$
\{V_{\mu,\tau_i}\mid i=0\text{, or }i=1,\dots,g+1 \text{ and } \tau_i\neq\mu\},
$$ 
where $\tau_0$ is the root and the indices $\tau_1,\dots,\tau_{g+1}$ are as in Definition \ref{pair}. In general, we 
may write
$$
\mathcal{SV}_\mu = \left\langle V_{\mu,\tau}\mid v_{\Gamma}(\tau)\le1\right\rangle.
$$
\end{prop}

\begin{proof}
If $\mu$ is the only vertex, then $\mathcal{SV}_\mu=\left\langle V_{\mu,\mu}\right\rangle$, and the claim is clear. Suppose that 
$\Gamma=\Gamma'(\eta,U)$. If there is $\eta'\in\Gamma$ different from $\eta$, for which $w_\Gamma(\eta')=1$, then we may find 
a graph $\Gamma''$ containing $\eta$, for which $\Gamma=\Gamma''(\eta',U')$. Thus we may in this situation choose 
$\eta\neq\mu$. But if $\eta\neq\mu$, then
$$
V_{\mu,\eta}=\sum_{\nu\sim\eta}V_{\mu,\nu},
$$
and so the value semigroup remains unchanged under the blowup. Hence we may assume 
that $\eta=\mu$ is the only vertex with weight one. This means that $\Gamma$ is the dual graph of 
the minimal principalization of the simple ideal corresponding to the vertex $\mu$, and the 
claim follows from Lemma \ref{Zar} below.
\end{proof}

Value semigroups are closely related to the jumping numbers. It follows from \cite[Theorem~6.2]{J}
(see also \cite[Lemma 6.1 and Remark 6.6]{J}) that if $\fa=\fp_\mu$ is simple, then $\xi$ 
is a jumping number of $\fa$ supported at $\mu$ if and only if 
$$
\xi V_{\mu,\mu}-q_{\mu,\gamma}-V_{\mu,\tau}\in \left\langle q_{\mu,\gamma}, V_{\mu,\tau}\right\rangle,
$$
where $(\gamma,\tau)$ is the pair associated to $\mu$ (Definition \ref{pair}). Our aim is to generalize this 
formula. Observe here that the first one of the two generators is not necessarily in $\mathcal{SV}_\mu$, 
but as we shall see in Proposition \ref{Sv} below, $q_{\mu,\gamma}$ is the greatest common divisor of the 
values $V_{\mu,i}$ with $i\in\Gamma_\gamma^\mu$, while $V_{\mu,\tau}$ is that of the values coming from the branch $\Gamma_\tau^\mu$. 

Let $\fa$ be a complete ideal in $R$ with a dual graph $\Gamma$. For vertices $\mu,\nu\in\Gamma$, let us 
define a submonoid $(\mathcal{SV}^\mu_\nu,+)$ of $\mathcal{SV}_\mu$ corresponding to the branch $\Gamma_\nu^\mu$ by setting
$$
\mathcal{SV}^\mu_\nu:=\left\langle V_{\mu,i}\mid i\in\Gamma_\nu^\mu\cup\{\mu\}\right\rangle.
$$
Clearly, $\mathcal{SV}^\mu_\nu=\mathcal{SV}^\mu_\eta$, if $\nu$ and $\eta$ define the same branch. Set $s^\mu_\nu:=\gcd\mathcal{SV}^\mu_\nu$ and write 
$(S^\mu,+)$ for the submonoid generated by these numbers, so that 
\begin{equation}\label{SSS}
S^\mu:=\left\langle s^\mu_\nu\mid\nu\in\Gamma\right\rangle
\end{equation}
Note that if $\mu$ is the only vertex in $\Gamma$, then $S^\mu=\left\langle s^\mu_\mu\right\rangle=\mathbb N$. Otherwise $S^\mu=\left\langle s^\mu_\nu\mid \nu\sim\mu\right\rangle$.

\begin{lem}\label{Zar}
Let $\Gamma$ be the dual graph of the minimal principalization of a simple ideal 
$\fp_\mu$ corresponding to the vertex $\mu$. Let $(\gamma_0,\tau_1),\dots,(\gamma_g,\tau_{g+1})=(\gamma,\tau)$ be the sequence of 
pairs associated to $\mu$ and write $\tau_0=\gamma_0$ for the root. Then
$$
\mathcal{SV}^\mu_\gamma =\left\langle V_{\mu,\tau_0},\dots,V_{\mu,\tau_g}\right\rangle\text{ and }
\mathcal{SV}^\mu_\tau=\left\langle V_{\mu,\tau}\right\rangle.
$$
Furthermore, we have $s^\mu_\gamma=q_{\mu,\gamma}$ and $s^\mu_\tau=V_{\mu,\tau}$, and $s^\mu_\gamma s^\mu_\tau=V_{\mu,\mu}$. The greatest integral 
multiple of $s^\mu_\gamma$ not in $\mathcal{SV}^\mu_\gamma$ is 
$$
\sum_{i=1}^gV_{\mu,\gamma_i}-\sum_{i=0}^gV_{\mu,\tau_i}.
$$
\end{lem}

\begin{proof}
By Equation \eqref{PW} we have for $\nu\in\Gamma\smallsetminus\{\mu\}$
$$
w_{\Gamma}(\nu)V_{\mu,\nu}=\sum_{i\sim\nu}V_{\mu,i}.
$$
Subsequently, we observe that if $\eta$ is such that $d(\nu,\eta)\le1$ and a nonnegative integer $k$ 
divides $V_{\mu,i}$ for every $i\neq\eta$ with $d(\nu,i)\le1$, then $k$ divides $V_{\mu,\eta}$ also. It follows that 
$$
\left\langle V_{\mu,i}\mid i\in\Gamma_{\tau_{n}}^{\gamma_n}\cup\{\gamma_n\} \right\rangle=\left\langle V_{\mu,\tau_{n}} \right\rangle
$$
as semigroups for any $n=0,\dots,g+1$, ($\gamma_{g+1}:=\mu$), so that $V_{\mu,\tau_{n}}$ is the greatest common 
divisor of the set. Especially, $\mathcal{SV}_\tau^\mu=\left\langle V_{\mu,\tau}\right\rangle$ so that $s^\mu_\tau= V_{\mu,\tau}$. Let us define a submonoid 
$\mathcal W_n$ as
$$
\mathcal W_n:=\left\langle V_{\mu,i}\mid i\in\Gamma_{\tau_0}^{\gamma_{n}}\cup\{\gamma_{n}\}\right\rangle
$$
for $n=0,\dots,g+1$, so that $\mathcal{SV}_\gamma^\mu=\mathcal W_{g+1}$. Let $\Sigma_n$ stand for the greatest integral multiple 
of $\gcd\mathcal W_n$ not in $\mathcal W_n$. Clearly, $\mathcal W_0=\left\langle V_{\mu,\tau_0}\right\rangle$ and $\gcd\mathcal W_0=V_{\mu,\tau_0}$ and $\Sigma_0=-V_{\mu,\tau_0}$. To 
complete our proof it is enough to show that for $n=0,\dots,g$,
\begin{equation}\label{WW}
\mathcal W_{n+1}=\left\langle \mathcal W_n, V_{\mu,\tau_n}\right\rangle, \gcd\mathcal W_{n+1}=q_{\mu,\gamma_n}
	\text{ and }
\Sigma_{n+1}=\Sigma_n+V_{\mu,\gamma_{n}}-V_{\mu,\tau_n}<V_{\mu,\gamma_n}.
\end{equation}

At first, note that the vertices of the set $\left(\Gamma_{\tau_0}^{\gamma_{n+1}}\cup\{\gamma_{n+1}\}\right)\smallsetminus\left(\Gamma_{\tau_0}^{\gamma_{n}}\cup\Gamma_{\tau_n}^{\gamma_n}\right)$ yield a 
path $\gamma_{n}=\eta_0^n\sim\cdots\sim\eta^n_{k_n}=\gamma_{n+1}$ for every $n=0,\dots,g$. Thus if $n=0$, it follows 
from the observation we made at the beginning of our proof that $V_{\mu,\tau_0}$ divides $V_{\mu,\eta^0_1}$ and 
subsequently $V_{\mu,\eta^0_i}$ for every $i=0,\dots,k_0$. Hence $\mathcal W_1=\left\langle V_{\mu,\tau_0}\right\rangle=\left\langle\mathcal W_0, V_{\mu,\tau_0} \right\rangle$ and the 
greatest common divisor of the set is $q_{\mu,\gamma_0}= V_{\mu,\tau_0}$. Moreover, $\Sigma_1=\Sigma_0+V_{\mu,\gamma_0}-V_{\mu,\tau_0}=-V_{\mu,\gamma_0}$ 
so that $\Sigma_1<V_{\mu,\gamma_0}$.

Assume that Equation \eqref{WW} holds if $n<n_0$ for some $n_0\in\{1,\dots,g\}$. Suppose then 
that $n=n_0$. Recall that by \cite[Lemma 6.1]{J} (see also \cite[Remark 6.6]{J}) we have 
\begin{equation}\label{gcd1}
\gcd\{q_{\mu,\gamma_{n-1}},V_{\mu,\tau_{n}}\}=q_{\mu,\gamma_{n}}.
\end{equation}
Again, by the observation at the beginning of the proof, we see that $q_{\mu,\gamma_{n}}$ divides $V_{\mu,\nu}$ 
for every $\nu\neq\eta_1^n$ with $d(\gamma_n,\nu)\le1$, and therefore it divides also $V_{\mu,\eta_1^n}$. Subsequently, it 
divides every $V_{\mu,\eta_i^n}$ with $i=0,\dots,k_n$. This shows that $\gcd\mathcal W_{n+1}=q_{\mu,\gamma_n}.$

Let us next verify that $\Sigma_{n+1}=\Sigma_{n}+V_{\mu,\gamma_n}-V_{\mu,\tau_n}$ is the greatest integral multiple 
of $q_{\mu,\gamma_n}$ not in $\left\langle \mathcal W_n, V_{\mu,\tau_n}\right\rangle$. By \cite[Corollary 3.16]{J} we know that $q_{\mu,\gamma_{n-1}}=q_{\gamma_{n},\gamma_{n-1}}q_{\mu,\gamma_{n}}$. 
Therefore by \cite[Lemma 6.1]{J} we get $V_{\mu,\gamma_{n}}=q_{\gamma_{n},\gamma_{n-1}}V_{\mu,\tau_{n}}$, and so
$$
\Sigma_{n+1}=\Sigma_{n}+(q_{\gamma_n,\gamma_{n-1}}-1)V_{\mu,\tau_{n}},
$$
which is clearly an integral multiple of $q_{\mu,\gamma_n}$.

Let $m\in\mathbb N$ be such that $\Sigma_{n+1}+mq_{\mu,\gamma_n}$ is in $\left\langle \mathcal W_n, V_{\mu,\tau_n}\right\rangle$. Equivalently, we may 
write $\Sigma_{n+1}+mq_{\mu,\gamma_n}=s+tV_{\mu,\tau_{n}}$ for some $s\in\mathcal W_n$ and $t\in\{0,\dots,q_{\gamma_n,\gamma_{n-1}}-1\}$. Now 
$u:=q_{\gamma_n,\gamma_{n-1}}-1-t$ belongs to the same set as $t$ and we may reformulate
$$
s=\Sigma_{n}+uV_{\mu,\tau_{n}}+mq_{\mu,\gamma_n}\in\mathcal W_{n}.
$$
This holds if and only if $uV_{\mu,\tau_{n}}+mq_{\mu,\gamma_n}$ is a positive multiple of $q_{\mu,\gamma_{n-1}}$. Observe that 
the map 
$$
\varphi:u\mapsto\frac{uV_{\mu,\tau_{n}}}{q_{\mu,\gamma_n}} \mod q_{\gamma_n,\gamma_{n-1}}
$$
is a bijection between the sets $\{0,\dots,q_{\gamma_n,\gamma_{n-1}}-1\}$ and $\mathbb Z_{q_{\gamma_n,\gamma_{n-1}}}$ following from the 
fact that $\gcd\{q_{\gamma_n,\gamma_{n-1}},V_{\mu,\tau_{n}}/q_{\mu,\gamma_n}\}=1$. Therefore we may always find an integer $u\in\{0,\dots,q_{\gamma_n,\gamma_{n-1}}-1\}$ 
so that $\varphi(u)=-m\mod q_{\gamma_n,\gamma_{n-1}}$, but then $uV_{\mu,\tau_n}+mq_{\mu,\gamma_n}$ is 
divisible by $q_{\mu,\gamma_{n-1}}$, and it is positive if and only if $m>0$. Thus $\Sigma_{n+1}$ is the greatest 
integral multiple of $q_{\mu,\gamma_{n}}$ which does not belong to $\left\langle \mathcal W_{n},V_{\mu,\tau_n}\right\rangle$.

Finally, let us verify that $\Sigma_{n+1}<V_{\mu,\gamma_n}$ and that $\mathcal W_{n+1}=\left\langle\mathcal W_n, V_{\mu,\tau_n} \right\rangle$. By \eqref{BC} we get 
$P_{\nu}V_\mu=q_{\mu,\nu}$, and since $PQ=1$ we see that $q_{\mu,\nu}=\sum_{i\succ\nu}q_{\mu,i}+\delta_{\mu,\nu}$ and subsequently 
$q_{\mu,\nu}>0$ if $\nu\subset\mu$. Hence $V_{\mu,\nu'}<V_{\mu,\nu}$ if $\nu'\prec\nu\subset\mu$. It follows that $V_{\mu,\nu'}<V_{\mu,\nu}$ if 
$\nu'\subset\nu\subset\mu$. Especially, $V_{\mu,\gamma_{n-1}}<V_{\mu,\tau_n}$, but then 
$$
\Sigma_{n+1}<V_{\mu,\gamma_{n-1}}+V_{\mu,\gamma_n}-V_{\mu,\tau_n}<V_{\mu,\gamma_{n}}\le V_{\mu,\eta^n_i}
$$
for every $i=0,\dots,k_n$. As $\Sigma_{n+1}$ is the greatest integral multiple of $q_{\mu,\gamma_n}$ not in 
$\left\langle\mathcal W_n, V_{\mu,\tau_n} \right\rangle$, we see that $V_{\mu,\eta^n_i}\in \left\langle\mathcal W_n, V_{\mu,\tau_n} \right\rangle$ for every $i=0,\dots,k_n$, but then 
$$
\left\langle\mathcal W_n, V_{\mu,\tau_n} \right\rangle\subset\mathcal W_{n+1}\subset\left\langle\mathcal W_n, V_{\mu,\tau_n} \right\rangle,
$$ 
as wanted. Thus Equation \eqref{WW} holds for any $n=0,\dots,g$, and the claim follows.
\end{proof}

\begin{prop}\label{Sv}
The submonoid $\mathcal{SV}^\mu_\nu$ is generated by those values $V_{\mu,i}$ where $i\in\Gamma_\nu^\mu$ is 
an end. More precisely, if $(\gamma,\tau)$ is the pair associated to $\mu$, then the following holds:
$$
\mathcal{SV}^\mu_\nu=
\begin{cases}
\left\langle V_{\mu,\mu}\right\rangle &\text{ if }\Gamma_\nu^\mu \text{ is posterior to } \mu;\\
\left\langle V_{\mu,\tau}\right\rangle &\text{ if }\nu \in \Gamma_\tau^\mu;\\
\left\langle V_{\mu,i}\mid v_{\Gamma}(i)=1, i\in\Gamma_\gamma^\mu\right\rangle &\text{ if }\nu \in 
\Gamma_\gamma^\mu.
\end{cases}
$$
We have $s^\mu_\nu=q_{\mu,\gamma}$ if $\nu\in\Gamma_\gamma^\mu$, otherwise $s^\mu_\nu$ is the generator of the submonoid $\mathcal{SV}^\mu_\nu$. For 
any $\nu\neq\mu$, the greatest integral multiple of $s^\mu_\nu$ not in $\mathcal{SV}^\mu_\nu$ is
$$
M_\nu=M^\mu_\nu:=\sum_{j\in\Gamma_\nu^\mu}(v_\Gamma(j)-2)V_{\mu,j},
$$
whereas $M_\mu:=-V_{\mu,\mu}$. 
\end{prop}

\begin{proof}
Suppose that $\eta\neq\mu$ is a vertex in $\Gamma$ with $w_\Gamma(\eta)=1$. Let $\Gamma'$ be such that
$\Gamma=\Gamma'(\eta,U)$ and suppose that the claim holds for $\Gamma'$. In the case $\nu\in\Gamma'$, let $\mathcal{SV'}^\mu_\nu$ denote 
the submonoid generated by the values $V_{\mu,i}$ where $i\in(\Gamma')^\mu_\nu\cup\{\mu\}$. Write $M'_\nu$ for the greatest 
integral multiple of $\gcd \mathcal{SV'}^\mu_\nu$ not in $\mathcal{SV'}^\mu_\nu$.
 
As $U$ consists of vertices adjacent to $\eta$ we observe that if $\eta\in\Gamma_\nu^\mu$, then $U\subset\Gamma_\nu^\mu\cup\{\mu\}$. 
By Equation \ref{PW} we then see that if $V_{\mu,i}\in \mathcal{SV'}^\mu_\nu$ for every $i\in U$, then $V_{\mu,\eta}\in \mathcal{SV'}^\mu_\nu$, as 
$w_\Gamma(\eta)=1$ and $\delta_{\mu,\eta}=0$. Thus the submonoid $\mathcal{SV}^\mu_\nu$ is the same as the submonoid $\mathcal{SV'}^\mu_\nu$ 
when $\nu\neq\eta$. Further, $\mathcal{SV}^\mu_\eta=\mathcal{SV}^\mu_\nu$ for any $\nu\in\Gamma_\eta^\mu$, and in the case $\Gamma_\eta^\mu=\{\eta\}$ we see 
that $\mathcal{SV}^\mu_\eta=\mathcal{SV'}^\mu_\mu$, as $\eta$ must be a free blowup of $\mu$, in which case Equation \eqref{PW} yields 
$V_{\mu,\eta}=V_{\mu,\mu}$. Subsequently, the greatest common divisor of the values from a branch 
remains unchanged under the blowup. Obviously, this is also the case with the greatest 
of its integral multiples not in the submonoid.

To see that the formula for $M_\nu$ holds, note first that, $M_\nu=M'_\nu$ if $\eta\notin\Gamma_\nu^\mu$. If $\Gamma_{\nu}^\mu=\{\eta\}$, 
then $v_{\Gamma}(\eta)=1$ and $M_\eta=(v_{\Gamma}(\eta)-2)V_{\mu,\eta}=-V_{\mu,\eta}$, as wanted, since $\mathcal{SV}^\mu_{\eta}=\left\langle V_{\mu,\eta}\right\rangle$ and 
$V_{\mu,\eta}=V_{\mu,\mu}$. Suppose then that $\eta$ is not the only vertex on $\Gamma_\nu^\mu$. As $M_\nu=M_{\nu'}$ for every 
$\nu'\in\Gamma_\nu^\mu$, we may assume that $\nu\neq\eta$. Then
$$
M_\nu=M'_\nu+\Delta_\eta,
$$
where
$$
\Delta_\eta:=(v_\Gamma(\eta)-2)V_{\mu,\eta}-\sum_{\mu\neq i\sim\eta}(v_{\Gamma'}(i)-v_{\Gamma}(i))V_{\mu,i}.
$$
If $\eta$ is free, then $v_{\Gamma}(\eta)=1$, but since there is only one vertex $i$ adjacent to it and $\eta$ is 
not the only vertex on the branch, we see that $i\neq\mu$, $v_{\Gamma}(i)=v_{\Gamma'}(i)+1$ and $V_{\mu,\eta}=V_{\mu,i}$ 
by Equation \eqref{PW}. Therefore $\Delta_\eta=0$. If $\eta$ is not free, then $v_{\Gamma}(\eta)=2$ and $v_{\Gamma}(i)=v_{\Gamma'}(i)$ 
for any $i\sim\eta$. Hence $\Delta_\eta=0$ in any case. This shows that also the numbers $M_\nu$ remain 
unchanged under the blowup.

By the above we may blow down $\eta$, and continuing this way, we may eventually 
assume that $w_{\Gamma}(i)>1$ for every vertex $i\neq\mu$. This is to say that $\Gamma$ is the dual graph of
the minimal principalization of a simple ideal $\fp_\mu$ corresponding to the vertex $\mu$, but this 
case is clear by Lemma \ref{Zar} above.
\end{proof}

\begin{rem}\label{Sgen2}
Because $V_{\mu,\mu}=q_{\mu,\gamma}V_{\mu,\tau}$, we see that $S^\mu=\left\langle q_{\mu,\gamma},V_{\mu,\tau}\right\rangle.$ Furthermore, 
$\gcd\{q_{\mu,\gamma},V_{\mu,\tau}\}=q_{\mu,\mu}$ by \eqref{gcd1}. On the other hand, $q_{\mu,\mu}=1$ as easily seen since $Q=P^{-1}$. 
It follows that $S^\mu$ is always a numerical semigroup, i.~e., a submonoid of $\bbN$ with finite 
complement.
\end{rem}

\section{Main result}

\begin{thm}\label{2}
A positive number $\xi$ is a jumping number in $\mathcal H_\mu^\fa$ exactly when
$$
h^\fa_\mu(\xi):=d_\mu\xi
+(v_\Gamma(\mu)-2)V_{\mu,\mu}-\sum_{\nu\sim\mu}s^\mu_\nu\left\lceil\sum_{i\in\Gamma_\nu^\mu}\frac{\widehat d_i V_{\mu,i}}{s^\mu_\nu}\,\xi\right\rceil^{+}\in S^\mu,
$$
where $S^\mu$ is the submonoid of $\mathbb N$ defined by Equation \eqref{SSS} and $\left\lceil\:\: \right\rceil^{+}$ means rounding up to 
the nearest positive integer.
\end{thm}

\begin{rem}
This result yields a formula for the set of the jumping numbers of $\fa$ supported at $\mu$:
$$
\cH_\mu^\fa =\left\{\frac{t}{d_\mu}\middle | t+(v_{\Gamma}(\mu)-2)V_{\mu,\mu}
-\sum_{\nu\sim\mu}s^\mu_\nu\left\lceil t\sum_{i\in\Gamma_\nu^\mu}\frac{\widehat d_i V_{\mu,i}}{	s^\mu_\nu d_\mu}\right\rceil^{+}\in S^\mu\right\}.
$$
\end{rem}

\begin{rem}
It easily follows from Proposition \ref{Sv} that the numbers $s^\mu_\nu$ and $V_{\mu,\mu}$ present 
in the above formula generate $S^\mu$. Moreover,$$s^\mu_\nu=\gcd\left\{V_{\mu,\tau}\mid v_\Gamma(\tau)=1\text{ and }\tau\in\Gamma_\nu^\mu\right\}$$ for 
every $\nu\sim\mu$.
\end{rem}

\begin{rem}
By Proposition \ref{Sv} $V_{\mu,\mu}\in S^\mu$ and $s^\mu_\nu=V_{\mu,\mu}$ for any $\nu$ such that $\Gamma^\mu_\nu$ is 
posterior to $\mu$. It follows that, for any positive integers $t_\nu$,
$$
\sum_{\nu\sim\mu}s^\mu_\nu t_\nu-(v_\Gamma(\mu)-2)V_{\mu,\mu}\in S^\mu.
$$
Subsequently, if $\xi\in\cH_\mu^\fa$, then $d_\mu\xi\in S^\mu$. Note that the converse is not true: if for 
example the ideal in question is the maximal ideal and $\xi=1$, then $\mu=\tau_0$ is the only 
vertex of the dual graph of our ideal and $d_\mu=1$. Clearly, $S^\mu=\mathbb N$ so that $d_\mu\xi\in S^\mu$, but 
as well known, $1$ is not a jumping number of the maximal ideal, see e.~g. Example \ref{Ex0} 
below.
\end{rem}

\begin{proof}[Proof of Theorem \ref{2}]
Note first that the case $\mu$ is the only vertex of the dual graph is 
trivial and is dealt with in Example \ref{Ex0} below. Thus we may assume $\mu$ has adjacent 
vertices. Obviously, by Remark \ref{DualG} we could also choose any dual graph containing $\mu$ and 
having other vertices, too. To begin with, observe that
$$
h^\fa_\mu(\xi)=\frac{nd_\mu\xi}{n}+(v_{\Gamma}(\mu)-2)V_{\mu,\mu}
-\sum_{\nu\sim\mu}s^\mu_\nu\left\lceil\sum_{i\in\Gamma_\nu^\mu}\frac{n\widehat d_i V_{\mu,i}}{s^\mu_\nu}\cdot\frac{\xi}{n}\right\rceil^{+}\hspace{-3pt}
=h^{\fa^n}_\mu\hspace{-3pt}\left(\textstyle\frac{\xi}{n}\right)
$$
Hence $h^\fa_\mu(\xi)\in S^\mu$ exactly, when $h^{\fa^n}_\mu\hspace{-3pt}\left(\textstyle\frac{\xi}{n}\right)\in S^\mu$. Together with Remark \ref{nUn} this shows 
that, by considering $\fa^n$ with $n$ big enough, we may assume $\widehat d_{\langle\widehat d\,\rangle}\in\mathbb N^\Gamma$.

Let $\fb$ be the ideal having the factorization vector $\widehat d_{\langle\widehat d\,\rangle}$. Note that $\fb$ is a product of 
ideals $\fp_\nu$ where $d(\nu,\mu)\le1$, and we may regard $\Gamma$ as a dual graph and $V$ as a valuation 
matrix of $\fb$ (see Remark \ref{DualG}). According to Lemma \ref{dM}
$$
(\widehat d_{\langle\widehat d\,\rangle} V)_\nu=d_\nu
	\text{ and }
\sum_{i\in\Gamma_\nu^\mu}\left(\widehat d_{\langle\widehat d\,\rangle}\right)_i V_{\mu,i}=\sum_{i\in\Gamma_\nu^\mu}\widehat d_iV_{\mu,i} 
$$
for $\nu$ with $d(\mu,\nu)\le1$. This shows, together with Lemma \ref{b-a}, that we may assume $\fa=\fb$, 
i.~e., $\widehat d_\nu=0$ unless $d(\mu,\nu)\le1$. Thereby 
$$
\sum_{i\in\Gamma_\nu^\mu}\frac{\widehat d_i V_{\mu,i}}{s^\mu_\nu}\,\xi=\frac{\widehat d_\nu V_{\mu,\nu}}{s^\mu_\nu}\,\xi
$$
for every $\nu\sim\mu$.

Set $\widehat g:=\xi\widehat d+\widehat\kappa$, where $\widehat\kappa:=(v_\Gamma(\nu)-2)_{\nu\in\Gamma}$. Observe that because
$$
(1,\dots,1)P^\textsc t P=(w_{\Gamma}(\nu)-v_{\Gamma}(\nu))_{\nu\in\Gamma}
$$ 
and since $\widehat k=(2-w_{\Gamma}(\nu))_{\nu\in\Gamma}$ by \eqref{K}, we get $k_\nu+1=-(\widehat\kappa V)_\nu$. Subsequently, we obtain
$$
\lambda(\widehat g\widehat E, dE;\nu)=\frac{\xi d_\nu+(\widehat\kappa V)_\nu+k_\nu+1}{d_{\nu}}=\xi 
$$
for any $\nu\in\Gamma$. Consider a vector$\widehat \phi:= \widehat g_{\langle\widehat\kappa\rangle[\zeta]}$where $\zeta_\nu:=0$ unless $\nu\sim\mu$ in which case 
\begin{equation*}
\label{KD}
\zeta_\nu
:=\frac{s^\mu_\nu}{V_{\mu,\nu}}\left(\left\lceil\frac{\widehat d_\nu V_{\mu,\nu}}{s^\mu_\nu}\,\xi\right\rceil^+
-\frac{\widehat d_\nu V_{\mu,\nu}}{s^\mu_\nu}\,\xi\right).
\end{equation*}
Let $M_\nu$ be as in Proposition \ref{Sv}. By using Lemma \ref{dM} a direct calculation shows that
\begin{equation}\label{phihmu}
\widehat\phi_\mu=\xi\widehat d_\mu+(v_\Gamma(\mu)-2)-\sum_{\nu\sim\mu}\rho_{[\mu,\nu]}(\mu)\zeta_\nu=\frac{h^\fa_\mu(\xi)}{V_{\mu,\mu}}
\end{equation}
and 
\begin{equation}\label{phnuD}
\widehat\phi_\nu
=\xi\widehat d_\nu+\zeta_\nu+\sum_{i\in\Gamma_\nu^\mu}\rho_{[\nu,i]}(\mu)(v_\Gamma(i)-2)
=\frac{M_\nu}{V_{\mu,\nu}}+\frac{s^\mu_\nu}{V_{\mu,\nu}}\left\lceil\frac{\widehat d_\nu V_{\mu,\nu}}{s^\mu_\nu}\xi\right\rceil^+.
\end{equation}
Observe that since $\widehat\phi=(\xi\widehat d+\widehat\kappa_{\langle\widehat \kappa\,\rangle})_{[\zeta]}$, we have $\widehat\phi_\nu=0$ when $d(\mu,\nu)>1$. Furthermore, 
as $\zeta_\nu\ge0$ where the equality may take place only if $\widehat d_\nu>0$, we know by Lemma \ref{dM} that 
$(\widehat\phi V)_\mu=(\widehat g V)_\mu$ and
$$
(\widehat\phi V)_\nu\ge(\widehat g V)_\nu
$$
for every $\nu$ with $\nu\sim\mu$, where the equality may take place only if $\widehat d_\nu>0$.

Suppose now that $\xi\in\mathcal H_\mu^\fa$. By \cite[Theorem 1 and Lemma 6]{HJ} we have $\widehat f\in\mathbb N^\Gamma$ satisfying 
$\widehat f_\nu>0$ only if $\nu$ is an end different from $\mu$, and further, $\lambda((\widehat f V)_\mu,\mu)=\xi$ and $\lambda((\widehat f V)_\nu,\nu)$ 
is increasing on every path going away from $\mu$. Because also $\lambda((\widehat\phi V)_\mu,\mu)=\lambda((\widehat g V)_\mu,\mu)=\,\xi$, 
we get by applying Lemma \ref{dM0} $\widehat\phi^\mathcal N=\widehat f^\mathcal N$, and further
$$
\widehat\phi^\mathcal N_{\mu}V_{\mu,\mu}
=\left(\widehat\phi^\mathcal N V\right)_{\mu}
=\left(\widehat f^\mathcal N V\right)_{\mu}
=\sum_{\nu\sim\mu}\sum_{i\in\Gamma_\nu^\mu}\widehat f_i V_{\mu,i}.
$$
Thus Equation \eqref{phihmu} gives
$$
h^\fa_\mu(\xi)=\left(\widehat\phi^\mathcal N_{\mu}-\sum_{\nu\sim\mu}\rho_{[\mu,\nu]}(\mu)\widehat\phi_\nu\right) V_{\mu,\mu}
=\sum_{\nu\sim\mu}\left(\sum_{i\in\Gamma_\nu^\mu}\widehat f_i V_{\mu,i}-\widehat\phi_\nu V_{\mu,\nu}\right).
$$
On the other hand, according to Equation \eqref{phnuD} and Proposition \ref{Sv}, $\widehat\phi_\nu V_{\mu,\nu}$ is the least 
integral multiple of $s^\mu_\nu$, for which 
$$
\widehat\phi_\nu V_{\mu,\nu}\ge M_\nu+\widehat d_\nu V_{\mu,\nu}\xi,
$$
where the inequality is strict if $\widehat d_\nu=0$. Observe that Proposition \ref{Sv} yields
$$
M_\nu+\widehat d_\nu V_{\mu,\nu}\xi
=\sum_{j\in\Gamma_\nu^\mu}(v_\Gamma(j)-2)V_{\mu,j}+\widehat d_\nu V_{\mu,\nu}\xi
=(\widehat g_{\langle\widehat\kappa\rangle})_\nu V_{\mu,\nu},
$$ 
where the last equality follows from Lemma \ref{dM}. Since every $V_{\mu,i}$ is divisible by $s^\mu_\nu$ for 
$i\in\Gamma_\nu^\mu$, we see that $h^\fa_\mu(\xi)\in S^\mu$, if for every $\nu\sim\mu$ holds
\begin{equation}
\label{fph}
\sum_{i\in\Gamma_\nu^\mu}\widehat f_i V_{\mu,i}\ge(\widehat g_{\langle\widehat\kappa\rangle})_\nu V_{\mu,\nu}.
\end{equation}
Observe that in the case $\widehat d_\nu=0$ the right hand side is equal to $M_\nu$ which is not in $\mathcal{SV}^\mu_\nu$ 
by Proposition \ref{Sv}, while the left hand side clearly is in $\mathcal{SV}^\mu_\nu$. Thus the inequality must 
be strict in this case.

As we saw above, $\widehat g^\mathcal N = \widehat f^\mathcal N$ and $(\widehat g_{\langle\,\widehat\kappa\,\rangle})_i=0$ unless $d(\mu,i)\le1$. Therefore 
$\widehat g_{\langle\,\widehat\kappa\,\rangle}=\widehat f^\mathcal N_{[\,\widehat g_{\langle\,\widehat\kappa\,\rangle}\,]}$, and $\widehat f_{\langle\,\widehat f\,\rangle}=\widehat f^\mathcal N_{[\,\widehat f_{\langle\,\widehat f\,\rangle}\,]}$ follows from Remark \ref{Fuf}. Recall that for any $\nu\sim\mu$
$$
\lambda((\widehat fV)_\nu,\nu)\ge\xi=\lambda((\widehat gV)_\nu,\nu).
$$ 
Hence, by applying Lemma \ref{dM}, we obtain for any $\nu\sim\mu$
$$
\lambda((\widehat f^\mathcal N_{[\,\widehat f_{\langle\,\widehat f\,\rangle}\,]}V)_\nu,\nu)
=\lambda((\widehat f_{\langle\,\widehat f\,\rangle}V)_\nu,\nu)
 \ge\lambda((\widehat g_{\langle\,\widehat\kappa\,\rangle}V)_\nu,\nu)
=\lambda((\widehat f^\mathcal N_{[\,\widehat g_{\langle\,\widehat\kappa\,\rangle}\,]}V)_\nu,\nu),
$$
so that
\begin{align*}
\left(\widehat f^\mathcal N_{[\,\widehat f_{\langle\,\widehat f\,\rangle}\,]}V\right)_\nu 
&=\left(\widehat f^\mathcal NV\right)_{\nu}+\sum_{i\sim\mu}\left(\widehat f_{\langle\,\widehat f\,\rangle}\right)_i\varphi_\mu^{[\mu,i]}(\nu)V_{\mu,\nu}\\
 \ge
\left(\widehat f^\mathcal N_{[\,\widehat g_{\langle\,\widehat\kappa\,\rangle}\,]}V\right)_\nu
&=\left(\widehat f^\mathcal NV\right)_{\nu}+\sum_{i\sim\mu}\left(\widehat g_{\langle\,\widehat\kappa\,\rangle}\right)_i\varphi_\mu^{[\mu,i]}(\nu)V_{\mu,\nu}
\end{align*}
According to Proposition \ref{D} $\varphi_\mu^{[\mu,i]}(\nu)\ge0$ for $\nu\sim\mu$, where the equality takes place always 
if $\mu\sim i\neq\nu$. Subsequently,
$$
\left(\widehat f_{\langle\,\widehat f\,\rangle}\right)_\nu\varphi_\mu^{[\mu,\nu]}(\nu)V_{\mu,\nu}\\
 \ge
(\widehat g_{\langle\,\widehat\kappa\,\rangle})_\nu\varphi_\mu^{[\mu,\nu]}(\nu)V_{\mu,\nu}
$$
which yields Inequality \eqref{fph}. Thus $h^\fa_\mu(\xi)\in S^\mu$.

Suppose next that $h_\mu^\fa(\xi)\in S^\mu$. Then we may take such non-negative integers $m_\nu$ for 
$\nu\sim\mu$, that 
$$
h_\mu^\fa(\xi)=\sum_{\nu\sim\mu}m_\nu s^\mu_\nu.
$$
Let $w:=\sum_{\nu\sim\mu}(m_\nu s^\mu_\nu/V_{\mu,\nu})\mathbf 1_\nu$ and define $\widehat\psi:=\widehat\phi_{\left[\,w\,\right]}$. Subsequently, we obtain by using 
Equation \eqref{phihmu} and Lemma \ref{dM}
$$
\widehat \psi_\mu=\frac{h_\mu^\fa(\xi)}{V_{\mu,\mu}}-\sum_{\nu\sim\mu}\frac{m_\nu s^\mu_\nu}{V_{\mu,\nu}}\frac{V_{\mu,\nu}}{V_{\mu,\mu}}=0,
$$
while for $\nu\sim\mu$ Equation \eqref{phnuD} together with Lemma \ref{dM} yields
$$
\widehat \psi_\nu
=\frac{s^\mu_\nu}{V_{\mu,\nu}}\left\lceil\frac{\widehat d_\nu V_{\mu,\nu}}{s^\mu_\nu}\,\xi\right\rceil^+
+\frac{M_\nu}{V_{\mu,\nu}}+\frac{m_\nu s^\mu_\nu}{V_{\mu,\nu}}
$$ 
Clearly, $\widehat\psi_\nu=0$ unless $\nu\sim\mu$, and further, since $s^\mu_\nu$ divides $\widehat\psi_\nu V_{\mu,\nu}$ and $\widehat\psi_\nu V_{\mu,\nu}>M_\nu$ 
we observe by Lemma \ref{Sv} that $\widehat\psi_\nu V_{\mu,\nu}\in \mathcal{SV}^\mu_\nu$ for every $\nu\sim\mu$. Therefore we may find 
$\widehat f\in\mathbb N^\Gamma$ with $\widehat f_i>0$ only if $v_\Gamma(i)=1$ and $i\neq\mu$ satisfying
$$
\left(\widehat f_{\langle\,\widehat f\,\rangle}\right)_\nu=\widehat\psi_\nu
$$
for every $\nu\sim\mu$. It follows from Lemma \ref{dM} that for any $\nu$ with $d(\mu,\nu)\le1$
$$
\lambda((\widehat fV)_\nu,\nu)=\lambda((\widehat\psi V)_\nu,\nu)\ge\lambda((\widehat\phi V)_\nu,\nu)\ge \xi,
$$ 
where the equality holds for $\nu=\mu$ and otherwise it may take place only if $\widehat d_\nu>0$. 
Subsequently, by choosing $a_i=\widehat fV_i$ for every $i$ with $d(i,\mu)\le1$, we may by using 
\cite[Lemma~5 and Lemma~6]{HJ} achieve a connected set $U\subset\{\nu\in\Gamma\mid d(\nu,\mu)\le1\}$ and non-negative 
integers $a_i$ for every $i\in\Gamma$ with $d(i,U)\le1$ satisfying the conditions of 
\cite[Theorem~1]{HJ}. Thereby $\xi\in\mathcal H_\mu^\fa$.
\end{proof}

\section{Examples}

\subsubsection*{Cases of low valence}

\begin{exmp}\label{Ex0}
In the case $v_{\Gamma}(\mu)=0$ we have only one vertex $\mu$. Moreover, $\widehat d_\mu>0$ while 
$V_{\mu,\mu}=1$ and thereby $d_{\mu}=\widehat d_\mu$. As the set of vertices adjacent to $\mu$ is empty, the set $S^\mu$ 
is $\mathbb N$. The claim of Theorem \ref{2} now says that $\xi$ is a jumping number if and only if
$$
\widehat d_\mu\xi-2\in\mathbb N.
$$
But this already follows from \cite[Theorem 6.2]{J} and Remark \ref{nUn}.
\end{exmp}

\begin{exmp}\label{Ex1}
Suppose $v_{\Gamma}(\mu)=1$. Let $\nu$ be the vertex adjacent to $\mu$. Then
$$
d_\mu\xi=\widehat d_\mu V_{\mu,\mu}\xi+\sum_{i\in\Gamma_\nu^\mu}s^\mu_\nu\frac{\widehat d_iV_{\mu,i}}{s^\mu_\nu}\xi.
$$
Since $(v_\Gamma(\mu)-2)V_{\mu,\mu}=-V_{\mu,\mu}$, we get by Theorem \ref{2} 
$$
h_\mu^\fa(\xi)=d_\mu\xi-V_{\mu,\mu}-s^\mu_\nu\left\lceil\sum_{i\in\Gamma_\nu^\mu}\frac{\widehat d_iV_{\mu,i}}{s^\mu_\nu}\xi\right\rceil^+
$$
Putting these together shows that 
$$
h_\mu^\fa(\xi)
=(\widehat d_\mu\xi-1)V_{\mu,\mu}
+s^\mu_\nu\sum_{i\in\Gamma_\nu^\mu}\frac{\widehat d_iV_{\mu,i}}{s^\mu_\nu}\xi
-s^\mu_\nu\left\lceil\sum_{i\in\Gamma_\nu^\mu}\frac{\widehat d_iV_{\mu,i}}{s^\mu_\nu}\xi\right\rceil^+ <0
$$ 
always unless $\widehat d_\mu\xi\ge1$. Especially, $\mathcal H_\mu^\fa$ is empty if $\widehat d_\mu=0$.
\end{exmp}

\begin{exmp}\label{Ex2}
Suppose $v_{\Gamma}(\mu)=2$. Since $(v_\Gamma(\mu)-2)V_{\mu,\mu}=0$, we get
\begin{align*}
h_\mu^\fa(\xi)
&=d_\mu\xi-\sum_{\nu\sim\mu}s^\mu_\nu\left\lceil\sum_{i\in\Gamma_\nu^\mu}\frac{\widehat d_iV_{\mu,i}}{s^\mu_\nu}\xi\right\rceil^+\\
&=\sum_{i\in\Gamma}\widehat d_i V_{\mu,i}\xi
	-\sum_{\nu\sim\mu}s^\mu_\nu\left\lceil\sum_{i\in\Gamma_\nu^\mu}\frac{\widehat d_iV_{\mu,i}}{s^\mu_\nu}\xi\right\rceil^+\\
&=\widehat d_\mu V_{\mu,\mu}\xi
	+\sum_{\nu\sim\mu}\sum_{i\in\Gamma_\nu^\mu}\widehat d_i V_{\mu,i}\xi
	-\sum_{\nu\sim\mu}s^\mu_\nu\left\lceil\sum_{i\in\Gamma_\nu^\mu}\frac{\widehat d_iV_{\mu,i}}{s^\mu_\nu}\xi\right\rceil^+\\
&=\widehat d_\mu V_{\mu,\mu}\xi
	+\sum_{\nu\sim\mu}s^\mu_\nu\left(\sum_{i\in\Gamma_\nu^\mu}\frac{\widehat d_i V_{\mu,i}}{s^\mu_\nu}\xi
	-\left\lceil\sum_{i\in\Gamma_\nu^\mu}\frac{\widehat d_iV_{\mu,i}}{s^\mu_\nu}\xi\right\rceil^+\right).
\end{align*}
Suppose $\xi\in\mathcal H_\mu^\fa$. Then $h_\mu^\fa(\xi)\ge0$, which implies that either $\widehat d_\mu>0$ or $\widehat d_\mu=0$ and
$$
\sum_{i\in\Gamma_\nu^\mu}\frac{\widehat d_i V_{\mu,i}}{s^\mu_\nu}\xi=\left\lceil\sum_{i\in\Gamma_\nu^\mu}\frac{\widehat d_iV_{\mu,i}}{s^\mu_\nu}\xi\right\rceil^+
$$
for each $\nu\sim\mu$. Note that the latter is possible only if for each $\nu\sim\mu$ there is such $i\in\Gamma_\nu^\mu$ 
that $\widehat d_i>0$.

Let us then assume that $\widehat d_\mu=0$, and let $\widehat f$ be such that 
\begin{equation}
\label{llx}
\lambda((\widehat f V)_\eta,\eta)\ge\lambda((\widehat f V)_\mu,\mu)=\xi.
\end{equation}
By Equation \eqref{PW} we have $w_\Gamma(\mu)V_{\mu,i}=\sum_{\nu\sim\mu}V_{\nu,i}+\delta_{\mu,i}$ for every $i\in\Gamma$. Moreover, by 
Equations \eqref{P} and \eqref{K}
$$
w_\Gamma(\mu)(k_\mu+1)-\sum_{\nu\sim\mu}(k_\nu+1)=\widehat k_\mu=2-v_\Gamma(\mu)=0.
$$
Thereby
\begin{align*}
	\lambda((\widehat f V)_\mu,\mu)
	&=\frac{(\widehat f V)_\mu+k_\mu+1}{(\widehat d V)_\mu}\\
	&=\frac{w_\Gamma(\mu)\left((\widehat f V)_\mu+k_\mu+1)\right)}{w_\Gamma(\mu)(\widehat d V)_\mu}\\
	&=\frac{\sum_{\nu\sim\mu}(\widehat f V)_\nu+\widehat f_\mu+\sum_{\nu\sim\mu}(k_\nu+1)}{\sum_{\nu\sim\mu}(\widehat d V)_\nu}\\
	&=\frac{(\widehat f V)_{\nu_1}+k_{\nu_1}+1+\widehat f_\mu+(\widehat f V)_{\nu_2}+k_{\nu_2}+1}
				 {\ \ (\widehat d V)_{\nu_1}\ \ \ \ \ \ \ \ +\ \ \ \ \ \ \ \ \ \ (\widehat d V)_{\nu_2}},
\end{align*}
where $\nu_1\sim\mu\sim\nu_2$. Furthermore, since we may assume that 
$$
\lambda((\widehat f V)_{\nu_1},\nu_1)\le\lambda((\widehat f V)_{\nu_2},\nu_2),
$$
this shows that
$$
\lambda((\widehat f V)_{\nu_1},\nu_1)\le\lambda((\widehat f V)_{\mu},\mu)\le\lambda(\widehat f_\mu+(\widehat f V)_{\nu_2},\nu_2),
$$
where the equality holds on the left if and only if it holds on the right. Putting these 
together with \eqref{llx} we observe that $\widehat f_\mu=0$ and
$$
\lambda((\widehat f V)_{\nu_1},\nu_1)=\xi=\lambda((\widehat f V)_{\nu_2},\nu_2).
$$
This is to say that both the vertices adjacent to $\mu$ support $\xi$. This means, informally 
speaking, that $\mu$ doesn't support jumping numbers independently. Especially, $\mathcal H_\mu^\fa$ is 
empty if there is $\nu$ such that $\widehat d_i=0$ whenever $i\in\Gamma_\mu^\nu$.
\end{exmp}

\subsubsection*{Simple ideals}

\begin{exmp}
Suppose $\fa$ is a simple ideal. If $\fa$ is the maximal ideal, then the dual graph 
consists of one vertex, but this case is already discussed in Example \ref{Ex0}. Thus we may 
assume $\fa$ is different from the maximal ideal. Since $\widehat d_i>0$ for only one vertex $i$, the 
examples \ref{Ex1} and \ref{Ex2} show that if $v_\Gamma(\mu)<3$ and $\widehat d_\mu=0$ then $\mathcal H_\mu^\fa$ is empty. Thus we may 
suppose that $\mu$ is a vertex with $\widehat d_\mu=1$ or $v_\Gamma(\mu)=3$. Note that if $\widehat d_\mu=1$, then $v_\Gamma(\mu)<3$ 
and if $v_\Gamma(\mu)=3$ then $\widehat d_\mu=0$.Let $(\gamma,\tau)$ be the pair associated to $\mu$.

Consider first the case $\widehat d_\mu=1$ and $v_\Gamma(\mu)=1$. Then $\Gamma_\gamma^\mu$ is the only branch emanating 
from $\mu$. By using Proposition \ref{Sv} we obtain $s^\mu_\gamma=q_{\mu,\gamma}$, and by applying equation $PQ=1$, 
we see that $q_{\mu,\gamma}=q_{\mu,\mu}=1$. Thus $S^\mu=\mathbb N$, and by Theorem \ref{2}, $\xi\in\mathcal H_\mu^\fa$ if and only if
$$
(\xi-1)d_\mu-1\in\mathbb N,\text{ i.~e., }\xi\in\left\{1+\frac{t+1}{d_\mu}\middle| t\in\mathbb N\right\}.
$$

Consider next the case $\widehat d_\mu=1$ and $v_\Gamma(\mu)=2$. It follows from Proposition \ref{Sv} that 
$s^\mu_\gamma=a:=q_{\mu,\gamma}$. Furthermore, $s^\mu_{\tau}=b:=V_{\mu,\tau}$ and $d_\mu=ab$. Now $S^\mu=\left\langle a, b\right\rangle$, and 
Theorem \ref{2} says that $\xi\in\mathcal H_\mu^\fa$ if and only if
$$
ab\xi-a-b\in \left\langle a, b\right\rangle,\text{ i.~e., }\xi\in\left\{\frac{s+1}{a}+\frac{t+1}{b}\middle| s,t\in\mathbb N\right\}.
$$

In the case $v_\Gamma(\mu)=3$ we see, again by Proposition \ref{Sv}, that $s^\mu_\gamma=a:=q_{\mu,\gamma}$ and 
$s^\mu_{\tau}=b:=V_{\mu,\tau}$ and $V_{\mu,\mu}=ab$, while $s^\mu_{\eta}=V_{\mu,\mu}$, where $\eta$ is the vertex corresponding to $\fa$. 
According to Theorem \ref{2} $\xi\in\mathcal H_\mu^\fa$ if and only if
$$
V_{\mu,\eta}\xi+V_{\mu,\mu}-a-b-V_{\mu,\mu}\left\lceil\frac{V_{\mu,\eta}\xi}{V_{\mu,\mu}}\right\rceil\in\left\langle a, b\right\rangle.
$$
Now $V_{\mu,\eta}/V_{\mu,\mu}=c:=q_{\eta,\mu}$, so the above is equivalent to
$$
\xi-\frac{\left\lceil c\xi\right\rceil-1}{c}=\frac{s+1}{ac}+\frac{t+1}{bc}\text{ for some }s,t\in\mathbb N.
$$ 
Obviously, the equation holds for $\xi+\frac{1}{c}$ if it holds for $\xi$. Subsequently, $\xi\in\mathcal H_\mu^\fa$ if and 
only if
$$
\xi\in\left\{\frac{s+1}{ac}+\frac{t+1}{bc}+\frac{m}{c}\middle|s,t,m\in\mathbb N,\frac{s+1}{ac}+\frac{t+1}{bc}\le\frac{1}{c}\right\}.
$$
Observe that $ac=q_{\eta,\gamma}$ and $bc=V_{\mu,\tau}q_{\eta,\mu}=V_{\eta,\tau}$ by, e.~g., \cite[Proposition 3.13]{J}. 

The above shows that Theorem \ref{2} gives an alternative proof of the formula for jumping 
numbers of a simple ideal (see \cite[Theorem 6.2]{J}).
\end{exmp}

\subsubsection*{General case}

\begin{exmp}\label{generalexample}
Let $\fa=\fp_1^2\fp_2^{\phantom{2}}\fp_3^2\fp_4^{\phantom{2}}\fp_5^3$ be an ideal, where
\begin{align*}
\fp_1&=\langle\overline{\phantom{^|}x^3y^3(x^3-y^2)\,,\,(x^3-y^2)^3+x^{11}\phantom{^|}}\rangle,\phantom{\sum^a}\\
\fp_2&=\langle\overline{\phantom{^|}x^2y^3         \,,\,(x^3-y^2)^2       \phantom{^|}}\rangle,\phantom{\sum^a}\\
\fp_3&=\langle\overline{\phantom{^|}xy^5           \,,\, x^3-y^7          \phantom{^|}}\rangle,\phantom{\sum^a}\\
\fp_4&=\langle\overline{\phantom{^|}x^{10}         \,,\,(x^3-(x-y)^2)^3   \phantom{^|}}\rangle,\phantom{\sum^a}\\
\fp_5&=\langle\overline{\phantom{^|}y^2            \,,\, x-y              \phantom{^|}}\rangle.\phantom{\sum^a}
\end{align*}
The dual graph of principalization of $\fa$ is as follows:

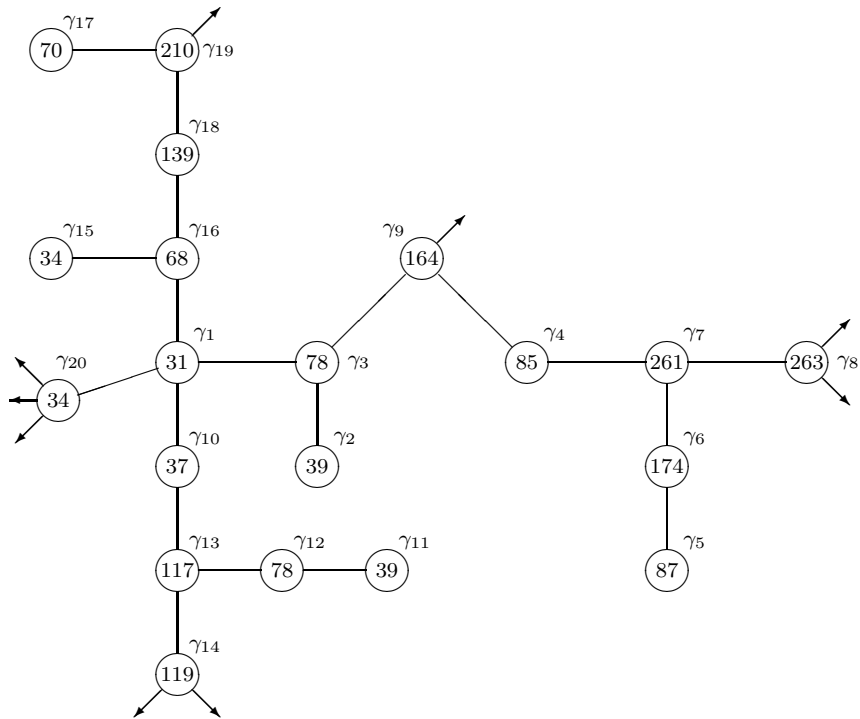
\unitlength=0.92mm
\begin{picture}(70,120)(-6,-27)
\put(-6,86){\makebox(0,0){\textbf{$\Gamma$}:}}
\put(3,24.5){\circle{6}}
\put(2,45){\circle{6}}
\put(2,75){\circle{6}}
\put(20,-15){\circle{6}}
\put(20,0){\circle{6}}
\put(20,15){\circle{6}}
\put(20,30){\circle{6}}
\put(20,45){\circle{6}}
\put(20,60){\circle{6}}
\put(20,75){\circle{6}}
\put(35,0){\circle{6}}
\put(40,15){\circle{6}}
\put(40,30){\circle{6}}
\put(50,0){\circle{6}}
\put(55,45){\circle{6}}
\put(70,30){\circle{6}}
\put(90,0){\circle{6}}
\put(90,15){\circle{6}}
\put(90,30){\circle{6}}
\put(110,30){\circle{6}}
\put(57.13,47.13){\vector(1,1){4}}
\put(22.13,77.13){\vector(1,1){4}}
\put(112.13,32.13){\vector(1,1){4}}
\put(112.13,27.87){\vector(1,-1){4}}
\put(0.87,26.63){\vector(-1,1){4}}
\put(0,24.5){\vector(-1,0){4}}
\put(0.87,22.37){\vector(-1,-1){4}}
\put(22.13,-17.13){\vector(1,-1){4}}
\put(17.87,-17.13){\vector(-1,-1){4}}
\put(23,0){\line(1,0){9}}
\put(20,-3){\line(0,-1){9}}
\put(20,3){\line(0,1){9}}
\put(38,0){\line(1,0){9}}
\put(20,27){\line(0,-1){9}}
\put(20,33){\line(0,1){9}}
\put(20,57){\line(0,-1){9}}
\put(20,63){\line(0,1){9}}
\put(23,30){\line(1,0){14}}
\put(73,30){\line(1,0){14}}
\put(93,30){\line(1,0){14}}
\put(40,27){\line(0,-1){9}}
\put(90,3){\line(0,1){9}}
\put(90,27){\line(0,-1){9}}
\put(17,75){\line(-1,0){12}}
\put(17,45){\line(-1,0){12}}
\put(42.13,32.13){\line(1,1){10.5}}
\put(67.87,32.13){\line(-1,1){10.5}}
\put(17.3,29.1){\line(-3,-1){11.5}}
\begin{tiny}
\put(3,24.5){\makebox(0,0){$34$}}
\put(2,45){\makebox(0,0){$34$}}
\put(2,75){\makebox(0,0){$70$}}
\put(20,-15){\makebox(0,0){$119$}}
\put(20,0){\makebox(0,0){$117$}}
\put(20,15){\makebox(0,0){$37$}}
\put(20,30){\makebox(0,0){$31$}}
\put(20,45){\makebox(0,0){$68$}}
\put(20,60){\makebox(0,0){$139$}}
\put(20,75){\makebox(0,0){$210$}}
\put(40,15){\makebox(0,0){$39$}}
\put(40,30){\makebox(0,0){$78$}}
\put(55,45){\makebox(0,0){$164$}}
\put(70,30){\makebox(0,0){$85$}}
\put(90,0){\makebox(0,0){$87$}}
\put(90,15){\makebox(0,0){$174$}}
\put(90,30){\makebox(0,0){$261$}}
\put(110,30){\makebox(0,0){$263$}}
\put(35,0){\makebox(0,0){$78$}}
\put(50,0){\makebox(0,0){$39$}}
\put(24,34){\makebox(0,0){$\gamma_1$}}
\put(44,19){\makebox(0,0){$\gamma_2$}}
\put(46,30){\makebox(0,0){$\gamma_3$}}
\put(74,34){\makebox(0,0){$\gamma_4$}}
\put(94,4){\makebox(0,0){$\gamma_5$}}
\put(94,19){\makebox(0,0){$\gamma_6$}}
\put(94,34){\makebox(0,0){$\gamma_7$}}
\put(116,30){\makebox(0,0){$\gamma_8$}}
\put(51,49){\makebox(0,0){$\gamma_9$}}
\put(24,19){\makebox(0,0){$\gamma_{10}$}}
\put(54,4){\makebox(0,0){$\gamma_{11}$}}
\put(39,4){\makebox(0,0){$\gamma_{12}$}}
\put(24,4){\makebox(0,0){$\gamma_{13}$}}
\put(24,-11){\makebox(0,0){$\gamma_{14}$}}
\put(6,49){\makebox(0,0){$\gamma_{15}$}}
\put(24,49){\makebox(0,0){$\gamma_{16}$}}
\put(6,79){\makebox(0,0){$\gamma_{17}$}}
\put(24,64){\makebox(0,0){$\gamma_{18}$}}
\put(26,75){\makebox(0,0){$\gamma_{19}$}}
\put(5,30){\makebox(0,0){$\gamma_{20}$}}
\end{tiny}
\end{picture}

\noindent
The factorization vector of $\fa$ is 
$$
(0,0,0,0,0,0,0,\mathbf{2},\mathbf{1},0,0,0,0,\mathbf{2},0,0,0,0,\mathbf{1},\mathbf{3}),
$$ 
and
\vspace{3pt}
$$\left[
\begin{smallmatrix}
\vspace{3pt}
\hspace{1.5pt}\phantom{\widehat t}1\hspace{4pt}	&\hspace{4pt}1\hspace{4pt}	&\hspace{4pt}2\hspace{4pt}	&\hspace{4pt}2
\hspace{4pt}	&\hspace{4pt}2\hspace{4pt}	&\hspace{4pt}4\hspace{4pt}	&\hspace{4pt}6\hspace{4pt}	&
\hspace{4pt}6\hspace{4pt}	&\hspace{4pt}4\hspace{4pt}	&\hspace{4pt}1\hspace{4pt}	&\hspace{4pt}1
\hspace{4pt}	&\hspace{4pt}2\hspace{4pt}	&\hspace{4pt}3\hspace{4pt}	&\hspace{4pt}3\hspace{4pt}	&
\hspace{4pt}1\hspace{4pt}	&\hspace{4pt}2\hspace{4pt}	&\hspace{4pt}2\hspace{4pt}	&\hspace{4pt}4
\hspace{4pt}	&\hspace{4pt}6\hspace{4pt}	&\hspace{4pt}1\hspace{4pt}\\
\vspace{3pt}1	&2	&3	&3	&3	&6	&9	&9	&6	&1	&1	&2	&3	&3	&1	&2	&2	&4	&6	&1\\
\vspace{3pt}2	&3	&6	&6	&6	&12	&18	&18	&12	&2	&2	&4	&6	&6	&2	&4	&4	&8	&12	&2\\
\vspace{3pt}2	&3	&6	&7	&7	&14	&21	&21	&13	&2	&2	&4	&6	&6	&2	&4	&4	&8	&12	&2\\
\vspace{3pt}2	&3	&6	&7	&8	&15	&22	&22	&13	&2	&2	&4	&6	&6	&2	&4	&4	&8	&12	&2\\
\vspace{3pt}4	&6	&12	&14	&15	&30	&44	&44	&26	&4	&4	&8	&12	&12	&4	&8	&8	&16	&24	&4\\
\vspace{3pt}6	&9	&18	&21	&22	&44	&66	&66	&39	&6	&6	&12	&18	&18	&6	&12	&12	&24	&36	&6\\
\vspace{3pt}6	&9	&18	&21	&22	&44	&66	&67	&39	&6	&6	&12	&18	&18	&6	&12	&12	&24	&36	&6\\
\vspace{3pt}4	&6	&12	&13	&13	&26	&39	&39	&26	&4	&4	&8	&12	&12	&4	&8	&8	&16	&24	&4\\
\vspace{3pt}1	&1	&2	&2	&2	&4	&6	&6	&4	&2	&2	&4	&6	&6	&1	&2	&2	&4	&6	&1\\
\vspace{3pt}1	&1	&2	&2	&2	&4	&6	&6	&4	&2	&3	&5	&7	&7	&1	&2	&2	&4	&6	&1\\
\vspace{3pt}2	&2	&4	&4	&4	&8	&12	&12	&8	&4	&5	&10	&14	&14	&2	&4	&4	&8	&12	&2\\
\vspace{3pt}3	&3	&6	&6	&6	&12	&18	&18	&12	&6	&7	&14	&21	&21	&3	&6	&6	&12	&18	&3\\
\vspace{3pt}3	&3	&6	&6	&6	&12	&18	&18	&12	&6	&7	&14	&21	&22	&3	&6	&6	&12	&18	&3\\
\vspace{3pt}1	&1	&2	&2	&2	&4	&6	&6	&4	&1	&1	&2	&3	&3	&2	&3	&3	&6	&9	&1\\
\vspace{3pt}2	&2	&4	&4	&4	&8	&12	&12	&8	&2	&2	&4	&6	&6	&3	&6	&6	&12	&18	&2\\
\vspace{3pt}2	&2	&4	&4	&4	&8	&12	&12	&8	&2	&2	&4	&6	&6	&3	&6	&7	&13	&20	&2\\
\vspace{3pt}4	&4	&8	&8	&8	&16	&24	&24	&16	&4	&4	&8	&12	&12	&6	&12	&13	&26	&39	&4\\
\vspace{3pt}6	&6	&12	&12	&12	&24	&36	&36	&24	&6	&6	&12	&18	&18	&9	&18	&20	&39	&60	&6\\
\vspace{3pt}1	&1	&2	&2	&2	&4	&6	&6	&4	&1	&1	&2	&3	&3	&1	&2	&2	&4	&6	&2
\end{smallmatrix}
\right]
$$
is the valuation matrix of $\fa$. Recall that by Proposition \ref{3d} any jumping number is supported 
at some vertex $\gamma$ with $\widehat d_\gamma>0$ or $v_{\Gamma}(\gamma)>2$. Therefore it is enough to consider 
the sets $\cH_{\gamma_j}^\fa$ for $j=1, 3, 7, 8, 9, 13, 14, 16, 19$ and $20$. By Theorem \ref{2} we know that
$$
\cH_{\gamma_1}^\fa=\left\{\frac{t}{31}\middle | t+(v_\Gamma(\gamma_1)-2)V_{\gamma_1,\gamma_1}-\sum_{\nu\sim\gamma_1}s^{\gamma_1}_\nu
\left\lceil t\sum_{i\in\Gamma_\nu^{\gamma_1}}\frac{\widehat d_i V_{\gamma_1,i}}{s^{\gamma_1}_\nu d_\mu}\right\rceil^{+}\in S^{\gamma_1}\right\}.
$$
Clearly, $v_\Gamma(\gamma_1)=4$ and $V_{\gamma_1,\gamma_1}=1$. Furthermore, we have $s^{\gamma_1}_\nu=1$ for every $\nu\sim\gamma_1$ so 
that $S^{\gamma_1}=\mathbb N$, and if we write $\Psi_{\gamma,\nu}:=\sum_{i\in\Gamma_\nu^{\gamma}}\widehat d_iV_{\gamma,i}$, we see that
$$
(\Psi_{\gamma_1,\nu})_{\nu=\gamma_3,\gamma_{10},\gamma_{16},\gamma_{20}}=\: \left(16,6,6,3\right).
$$
Subsequently,
$$
\cH_{\gamma_1}^\fa=\left\{\frac{t}{31}\middle | t+2-\left\lceil\frac{16\cdot t}{31}\right\rceil^{+}
-2\cdot\left\lceil\frac{6\cdot t}{31}\right\rceil^{+}-\left\lceil\frac{3\cdot t}{31}\right\rceil^{+}\ge0\right\}.
$$

In the case $j=3$ we see that $v_\Gamma(\gamma_3)=3$ and $V_{\gamma_3,\gamma_3}=6$. Moreover, 
$(s^{\gamma_3}_\nu)_{\nu=\gamma_1,\gamma_2,\gamma_9}=(2,3,6)$ so that $S^{\gamma_3}=\left\langle 2, 3\right\rangle$, and  
$$
(\Psi_{\gamma_3,\nu})_{\nu=\gamma_1,\gamma_2,\gamma_9}=(30,0,48).
$$
Since $\left\lceil 0\right\rceil^{+}=1$ we see that
\begin{align*}
\cH_{\gamma_3}^\fa
=&\left\{\frac{t}{78}\middle | t+6-2\cdot\left\lceil\frac{30\cdot t}{2\cdot78}\right\rceil^{+}-3\cdot\left\lceil\frac{0\cdot t}{3\cdot78}\right\rceil^{+}
-6\cdot\left\lceil\frac{48\cdot t}{6\cdot78}\right\rceil^{+}\in\mathbb N\smallsetminus\{1\}\right\}\\
=&\left\{\frac{t}{78}\middle | t+3-2\cdot\left\lceil\frac{30\cdot t}{2\cdot78}\right\rceil^{+}
-6\cdot\left\lceil\frac{48\cdot t}{6\cdot78}\right\rceil^{+}\in\mathbb N\smallsetminus\{1\}\right\}.
\end{align*}
Similarly,
\begin{align*}
\cH_{\gamma_7}^\fa=&\left\{\frac{t}{261}\middle | t+44-3\cdot\left\lceil\frac{129\cdot t}{3\cdot261}\right\rceil^{+}
	-66\cdot\left\lceil\frac{132\cdot t}{66\cdot261}\right\rceil^{+}\in \left\langle 3, 22\right\rangle\right\},\\
\cH_{\gamma_8}^\fa=&\left\{\frac{t}{263}\middle | t-67-\left\lceil\frac{129\cdot t}{263}\right\rceil^{+}\ge0\right\},\\
\cH_{\gamma_9}^\fa=&\left\{\frac{t}{164}\middle | t-2\cdot\left\lceil\frac{60\cdot t}{2\cdot164}\right\rceil^{+}
	-13\cdot\left\lceil\frac{78\cdot t}{13\cdot164}\right\rceil^{+}\in \left\langle 2,13\right\rangle\right\},\\
\cH_{\gamma_{13}}^\fa=&\left\{\frac{t}{117}\middle | t+14-3\cdot\left\lceil\frac{75\cdot t}{3\cdot117}\right\rceil^{+}
	-21\cdot\left\lceil\frac{42\cdot t}{21\cdot117}\right\rceil^{+}\in \left\langle 3, 7\right\rangle\right\},\\
\cH_{\gamma_{14}}^\fa=&\left\{\frac{t}{119}\middle | t-22-\left\lceil\frac{75\cdot t}{119}\right\rceil^{+}\ge0\right\},\\
\cH_{\gamma_{16}}^\fa=&\left\{\frac{t}{68}\middle | t+3-2\cdot\left\lceil\frac{50\cdot t}{2\cdot68}\right\rceil^{+}
	-6\cdot\left\lceil\frac{18\cdot t}{6\cdot68}\right\rceil^{+}\in \mathbb N\smallsetminus\{1\}\right\},\\
\cH_{\gamma_{19}}^\fa=&\left\{\frac{t}{210}\middle | t-20-3\cdot\left\lceil\frac{150\cdot t}{3\cdot210}\right\rceil^{+}
	\in\left\langle 3, 20\right\rangle\right\},
\end{align*}
and finally,
$$
\cH_{\gamma_{20}}^\fa=\left\{\frac{t}{34}\middle | t-2-\left\lceil\frac{28\cdot t}{34}\right\rceil^{+}\ge0\right\}.
$$
Thus we get 
\begin{align*}
\cH_{\gamma_1}^\fa =&\left\{\frac{t+10m}{31}+n\middle | t =3,4,5,7,8,9,10;\: m=0,1,2\right\}\cup\mathbb Z_{+},\\
\cH_{\gamma_3}^\fa =&\left\{\frac{5+10t+2m}{78}+n\middle | t,m,n\in\mathbb N;\:t<8;\: m<3-\frac{t}{4}\right\}\cup\mathbb Z_{+},\\
\cH_{\gamma_7}^\fa =&\left\{\frac{t+3m+129p}{261}+n\middle | t=46, 89; m,n\in\mathbb N; p=0, 1;\frac{t+3m+129p}{261}\le\frac{1+p}{2}\right\}\cup\mathbb Z_{+},\\
\cH_{\gamma_8}^\fa =&\left\{\frac{t+132}{263}\middle | t \in\mathbb N\right\},\\
\cH_{\gamma_9}^\fa =&\left\{\frac{19+21t+2m}{164}\middle | t,m\in\mathbb N\text{ and }\frac{3-t}{3}\le m\le4+\frac{16t}{5}\right\},\\
\cH_{\gamma_{13}}^\fa =&\left\{\frac{t+3m+57p}{117}+n\middle | t =22, 41;\,m,n\in\mathbb N;\,p=0, 1;\frac{t+3m+57p}{117}\le\frac{1+p}{2}\right\}\cup\mathbb Z_{+},\\
\cH_{\gamma_{14}}^\fa =&\left\{\frac{t+60}{119}\middle | t\in\mathbb N\right\},\\
\cH_{\gamma_{16}}^\fa =&\left\{\frac{t+2m}{68}+n\middle | t=11, 33, 55, 66; m=1, 2, 3, 4, 5, 6; n\in\mathbb N\text{ and }23\neq t+2m\le68\right\},\\
\cH_{\gamma_{19}}^\fa =&\left\{\frac{t+3m}{210}\middle | t = 71, 142, 210;\: m\in\mathbb N\right\},\\
\cH_{\gamma_{20}}^\fa=&\left\{\frac{t+12}{34}\middle | t\in\mathbb N\right\},
\end{align*}
and the set of jumping numbers of $\fa$ is
$$
\cH^\fa=\cH_{\gamma_{1}}^\fa\cup\cH_{\gamma_{3}}^\fa\cup\cH_{\gamma_{7}}^\fa\cup\cH_{\gamma_{8}}^\fa\cup\cH_{\gamma_{9}}^\fa
\cup\cH_{\gamma_{13}}^\fa\cup\cH_{\gamma_{14}}^\fa\cup\cH_{\gamma_{16}}^\fa\cup\cH_{\gamma_{19}}^\fa\cup\cH_{\gamma_{20}}^\fa.
$$
\end{exmp}

\end{document}